\documentclass[12pt]{extarticle}
\usepackage[utf8]{inputenc}
\usepackage[T1]{fontenc}

\usepackage{caption}
\oddsidemargin \evensidemargin
\usepackage{amssymb}
\usepackage{pmboxdraw}
\usepackage{amsmath}
\usepackage{bbm}
\usepackage{amsthm}
\usepackage{amssymb}
\usepackage{dsfont}
\usepackage[usenames,dvipsnames]{xcolor}
\usepackage{graphicx}
\usepackage{alphabeta}
\usepackage{tikz,pgfplots}
\pgfplotsset{compat=1.18}
\usepackage{pgfplots}
\usepackage[margin=0.5cm]{caption}
\usepackage{hyperref}
\usepackage{subcaption}
\usepackage[all]{xy}
\usepackage[greek,english]{babel}
\usepackage[export]{adjustbox}
\usepackage{hyperref}
\hypersetup{
    colorlinks=true,
    linkcolor=orange!80!black,
    filecolor=magenta,      
    urlcolor=RoyalBlue,
    citecolor=RoyalBlue,
}
\usepackage{enumitem}

\definecolor{myblue}{rgb}{0.00000,0.44700,0.74100}
\definecolor{myorange}{rgb}{0.8500, 0.3250, 0.0980}
\definecolor{myyellow}{rgb}{0.9290, 0.6940, 0.1250}
\definecolor{mypurple}{rgb}{0.4940, 0.1840, 0.5560}
\definecolor{mygreen}{rgb}{0.4660, 0.6740, 0.1880}

\usepackage{listings}

\newtheorem{theorem}{Theorem}[section]
\newtheorem{theorem*}[theorem]{Theorem*}

\newtheorem{lemma}[theorem]{Lemma}
\newtheorem{corollary}[theorem]{Corollary}
\newtheorem{proposition}[theorem]{Proposition}

\newtheorem{conjecture}[theorem]{Conjecture}

\theoremstyle{definition}
\newtheorem{examplex}[theorem]{Example}
\newenvironment{example}
{\pushQED{\qed}\examplex}
{\popQED\endexamplex}

\newtheorem{remark}[theorem]{Remark}

\theoremstyle{remark}

\newcommand{\Xcos}{{\cal X}_{A,\cos{}}}

\title{\bf Chebyshev Varieties}
\author{Za\"ineb Bel-Afia, Chiara Meroni and Simon Telen}
\date{}

\begin{document}

\maketitle

\begin{abstract}
\noindent Chebyshev varieties are algebraic varieties parametrized by Chebyshev polynomials or their multivariate generalizations. We determine the dimension, degree, singular locus and defining equations of these varieties. We explain how they play the role of toric varieties in sparse polynomial root finding, when monomials are replaced by Chebyshev polynomials. We present numerical root finding algorithms that exploit our results.
\end{abstract}

\section{Introduction}
The Chebyshev polynomials of the first kind are classically defined by the recurrence relation 
\begin{equation} \label{eq:chebpol1}
T_0(t) \, = \, 1, \quad T_1(t) \, = \, t, \quad \text{and} \quad T_{k+1}(t) \, = \, 2t \cdot T_{k}(t) - T_{k-1}(t). 
\end{equation}
They are often used in applied and numerical mathematics. For instance, Chebyshev expansions and interpolants are essential tools in function approximation \cite{trefethen2019approximation}. In short, one approximates a sufficiently well-behaved function $\phi: [a,b] \rightarrow \mathbb{R}$ by a Chebyshev~polynomial 
\begin{equation} \label{eq:funivcheb}
f(t) \, = \, c_0 + c_1 \, T_1(t) + \cdots + c_d \, T_d(t), 
\end{equation}
such that $\lVert f - \phi \rVert \leq \varepsilon$. For details, see \cite[Chapter 8]{trefethen2019approximation}. Finding the roots of $\phi$, i.e., the values $t \in [a,b]$ such that $\phi(t) = 0$, is replaced by the polynomial root finding problem $f(t) = 0$. 
This is the topic of \cite{boyd2013finding,day2005roots,noferini2017chebyshev}. Important for our story is the following fact, stated in more detail in the Introduction of \cite{noferini2017chebyshev}. One could expand $f(t)$ into the standard monomial basis
\[ f(t) \, = \, c_0 + c_1 \, T_1(t) + \cdots + c_n \, T_n(t) \, = \, e_0 + e_1 \, t + \cdots + e_n \, t^n \, = \, 0\]
and find the roots in $[a,b]$ from the new coefficients $e_0, e_1, \ldots, e_n$. For instance, these roots are among  the eigenvalues of the Frobenius companion matrix. It turns out this is a bad idea. Apart from the fact that the basis conversion $(c_i)_i \rightarrow (e_i)_i$ is ill-conditioned, the real roots of $f$ are much more sensitive to changes in the $e_i$ than to changes in the $c_i$. In numerical analysis terms, the real root finding problem in the Chebyshev basis is better conditioned. 

This role of Chebyshev polynomials in univariate root finding is our point of entry. We are concerned with the multivariate analog of solving $f(t) = 0$. We consider $m$ nonzero polynomials in $m$ variables $f_1, \ldots, f_m \in \mathbb{R}[t_1,\ldots,t_m]$. We seek to compute their common zeros in a bounded box $\square  \subset \mathbb{R}^m$, which we will take to be $[-1,1]^m$ without loss of generality:
\[ V_\square(f_1,\ldots,f_m) \, = \, \{ t \in \square=[-1,1]^m \, : \, f_1(t) \, = \, \cdots \, = \, f_m(t) \, = \, 0 \}. \]
While the univariate case $m = 1$ is straightforwardly turned into a linear algebra problem via the companion matrix, the case $m > 1$ is generally considered a problem of computational algebraic geometry \cite{cox2013ideals} or nonlinear algebra \cite{michalek2021invitation}. There, $\square$ is usually replaced by $\mathbb{C}^m$. Symbolic methods from computer algebra mostly use Gr\"obner bases \cite[Section 2]{cox2013ideals}. In numerical algebraic geometry, a more recent discipline which uses numerical methods to study algebraic varieties, the main methods are subdivision methods \cite{mourrain2009subdivision,parkinson2024chebyshev}, homotopy continuation \cite{sommese2005numerical,timme2021numerical} and numerical normal forms or resultants  \cite{dreesen2012back,nakatsukasa2015computing,sorber2014numerical,telen2020solving}. 

In analogy with \eqref{eq:funivcheb}, our polynomials $f_i(t)$ are assumed to be given in the following form: 
\begin{equation} \label{eq:fi}
 f_i(t) \, = c_{i,0} \, + \, c_{i,a_1} \, {\cal T}_{a_1}(t) \,  + \,  c_{i,a_2} \, {\cal T}_{a_2}(t) \, + \,  \cdots \, + \,  c_{i,a_n} \, {\cal T}_{a_n}(t),
\end{equation}
where $c_{i,0}, c_{i,a_j}$ are real numbers, the index $a_j = (a_{1j}, \ldots, a_{mj}) \in \mathbb{N}^m$ is an $m$-tuple of nonnegative integers,  and ${\cal T}_{a_j}$ are functions of $m$ variables which generalize the Chebyshev polynomials~\eqref{eq:chebpol1}. We discuss different choices for ${\cal T}_{a_j}$ below. This article studies the problem of computing $V_{\mathbb{C}^m}(f_1,\ldots,f_m) \supset V_\square(f_1,\ldots,f_m)$ from an algebro-geometric point of view. 

This leads us to define the objects in our title. For any fixed choice of functions ${\cal T}_{a_j}$, an integer matrix $A = [a_1 ~ \cdots ~ a_n] \in \mathbb{N}^{m \times n}$ defines a parametrization ${\cal T}_A : \mathbb{C}^m \rightarrow \mathbb{C}^n$ given by 
\begin{equation} \label{eq:phiA}
{\cal T}_A(t) \, = \, ( \, {\cal T}_{a_1}(t), \, {\cal T}_{a_2}(t), \, \ldots, {\cal T}_{a_n}(t) \, ). 
\end{equation}
When ${\cal T}_{a_j}$ is defined in one of the ways listed below, the Zariski closure of the image of ${\cal T}_A$ is called a \emph{Chebyshev variety}. We denote this variety by ${\cal X}_A = \overline{{\rm im} \, {\cal T}_A} \subset \mathbb{C}^n$. We observe that computing $V_{\mathbb{C}^m}(f_1,\ldots,f_m)$ with $f_i$ as in \eqref{eq:fi} can be reformulated as $m$ linear equations~on~${\cal X}_A$:
\begin{equation} \label{eq:lineqonX}
 c_0 + C \cdot x  = 0, \quad \text{and} \quad x \in {\cal X}_A,
 \end{equation}
where $C \in \mathbb{R}^{m \times n}$ is an $m \times n$ matrix with entries $C_{ij} = c_{i,a_j}$, $x = (x_1, \ldots, x_n)$ are coordinates on $\mathbb{C}^n$, and $c_0$ is the column vector $(c_{1,0}, \ldots, c_{m,0})^\top$. For computing $V_\square(f_1 ,\ldots, f_m)$, we replace ${\cal X}_A$ by the semialgebraic subset ${\cal X}_{A,\square} = {\cal T}_A(\square) \subset \mathcal{X}_A$ in \eqref{eq:lineqonX}. The above observation implies that geometric properties of ${\cal X}_A$, such as its dimension, degree and singular locus, govern the geometry and complexity of solving our equations. The goal in this paper is to describe these properties. When ${\cal T}_{a_j}(t) = t_1^{a_{1j}}t_2^{a_{2j}}\cdots t_m^{a_{mj}}$ are monomials instead, $\overline{{\rm im} \, {\cal T}_A}$ is an affine toric variety \cite{cox2011toric,telen2022introduction}. Our study of Chebyshev varieties for Chebyshev root finding mimics that of toric varieties for sparse root finding, see \cite[Example 1.1]{telen2022introduction}. We recall the toric approach in Section \ref{sec:2}, which will also clarify the relation between \eqref{eq:lineqonX} and $f_1 = \cdots = f_m = 0$. 

We now discuss our choices for the parametrizing functions ${\cal T}_{a_j}$ in \eqref{eq:phiA}. In the univariate case $(m = 1)$, ${\cal T}_{a_j} = T_{a_j}$ is the Chebyshev polynomial of degree $a_j$ or ${\cal T}_{a_j} = U_{a_j}$ is the $a_j^{th}$ Chebyshev polynomial of the second kind. This is the setting of Section \ref{sec:3}. For $m > 1$ we consider two different generalizations. Section \ref{sec:4} sets ${\cal T}_{a_j}(t_1) = T_{a_{1j}}(t)  \, T_{a_{2j}}(t_2) \, \cdots \,  T_{a_{mj}}(t_m)$. I.e., we use elements of the tensor product basis of $\mathbb{R}[t_1,\ldots,t_m] = \bigotimes_{i=1}^m \mathbb{R}[t_i]$ induced by the Chebyshev basis on each factor. Finally, in Section \ref{sec:5}, we set ${\cal T}_{a_j}(t) = \cos(a_j \cdot u)$, where $u = (u_1, \ldots, u_m)$ satisfies $\cos(u_i) = t_i$. Here $a_j \cdot u$ is the usual inner product. This generalizes the well-known property $T_k(t) = \cos(k \, {\rm acos}(t))$ for $t \in [-1,1]$. Figure \ref{fig:threesurfaces} shows that these different choices of parametrizations lead to varieties with different geometric properties. For a fixed $A \in \mathbb{N}^{2 \times 3}$, it plots some real points of the toric variety ${\cal Y}_A$, the tensor product Chebyshev variety ${\cal X}_{A,\otimes}$, and the Chebyshev variety ${\cal X}_{A,\cos{}}$ obtained from cosines.
In Section \ref{sec:6-4} we briefly discuss generalized constructions from root systems, see for instance \cite{hubert2022sparse,ryland2010multivariate}. 

 \begin{figure}[ht]
\begin{subfigure}{.32\textwidth}
  \centering
  \includegraphics[height = 4cm]{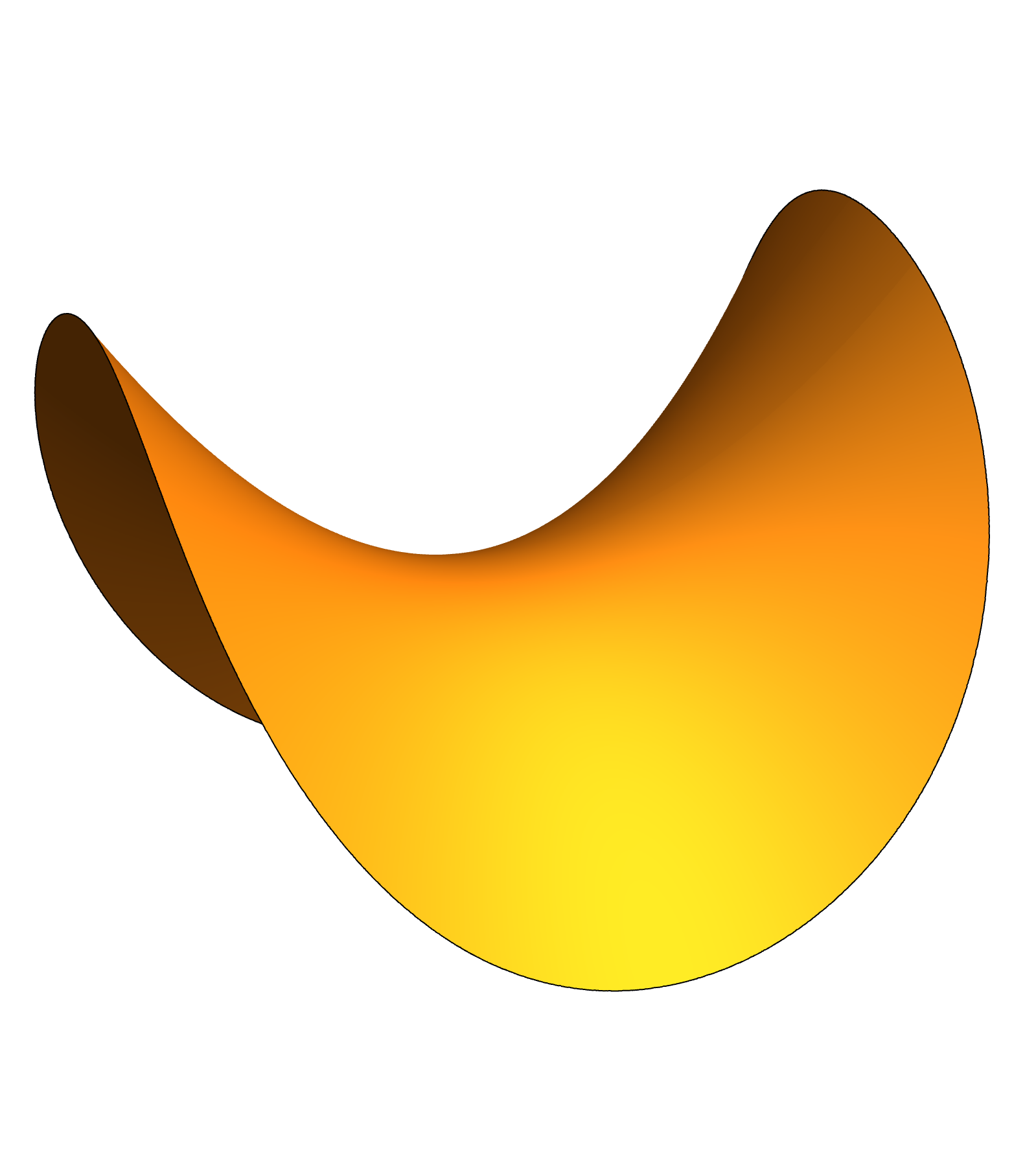} 
  \caption{toric variety ${\cal Y}_A$}
  \label{fig:sub-first}
\end{subfigure}
\begin{subfigure}{.32\textwidth}
  \centering
  \includegraphics[height = 4cm]{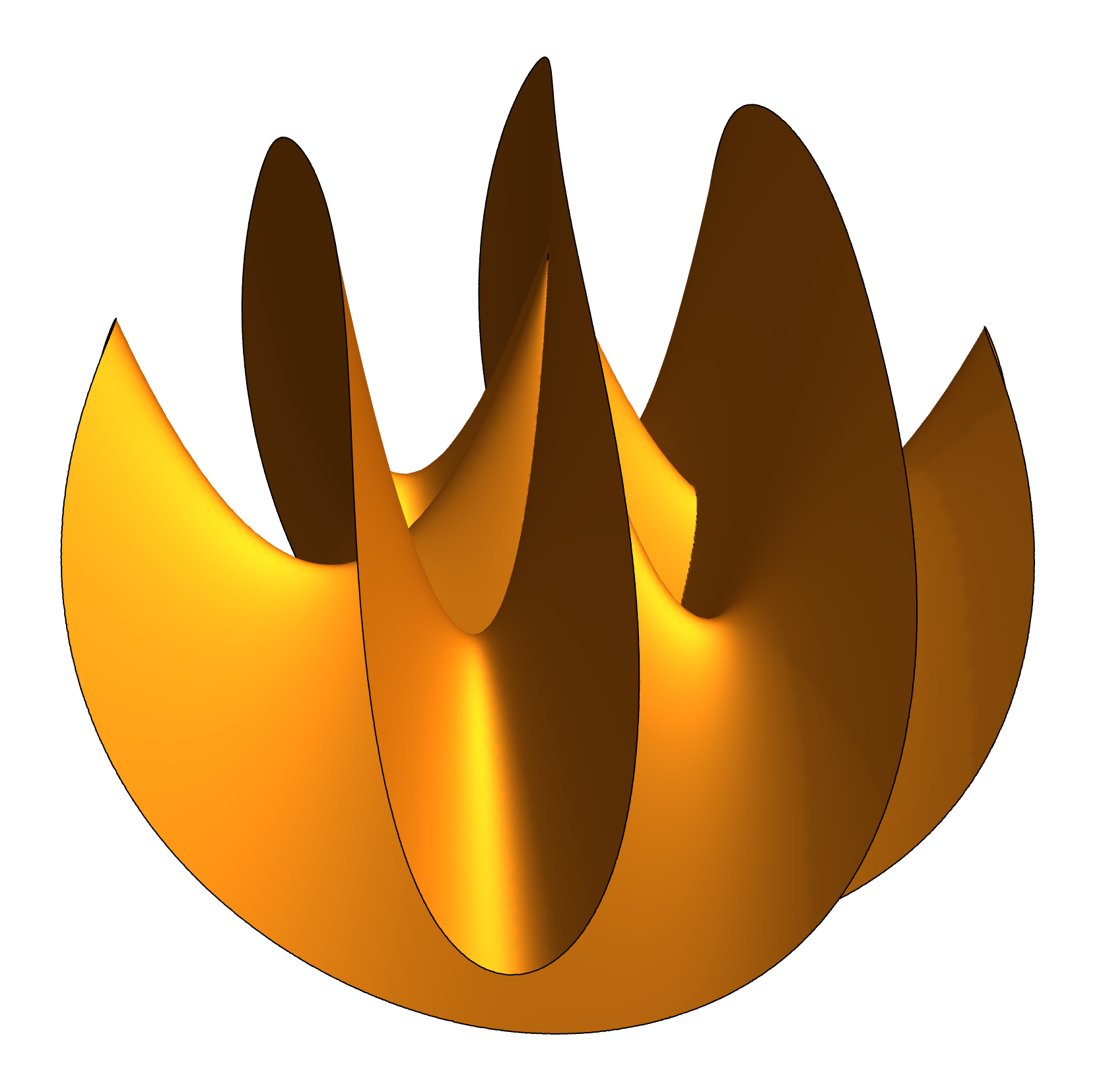}  
  \caption{Chebyshev variety ${\cal X}_{A,\otimes}$}
  \label{fig:sub-second}
\end{subfigure}
\begin{subfigure}{.32\textwidth}
  \centering
  \includegraphics[height = 4cm]{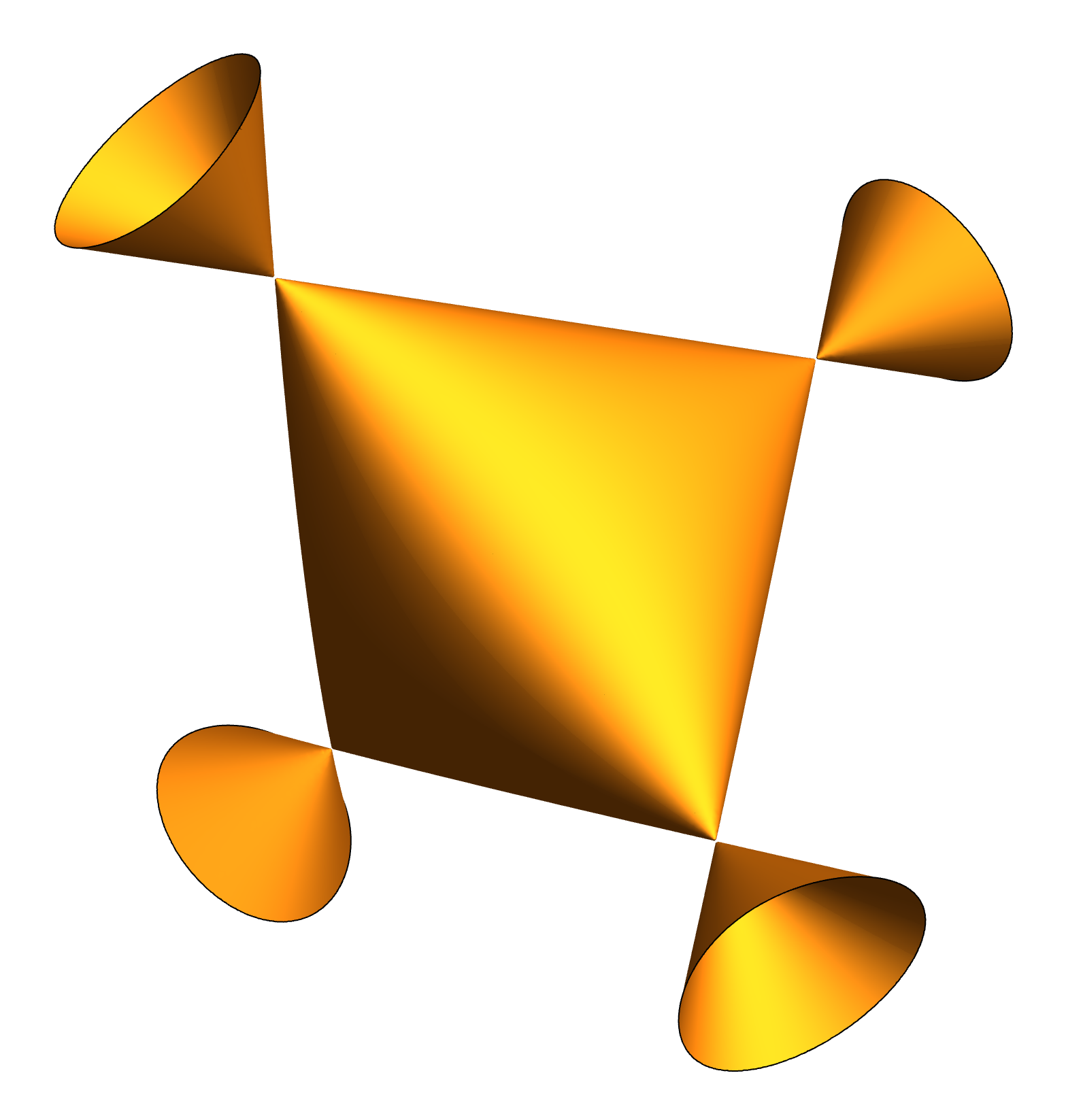}  
  \caption{Chebyshev variety ${\cal X}_{A,\cos{}}$}
  \label{fig:sub-third}
\end{subfigure}
\caption{Three surfaces obtained from the matrix $A = \left [ \begin{smallmatrix}
    1 & 1 & 2 \\ 2 & 1 & 3
\end{smallmatrix} \right ]$.}
\label{fig:threesurfaces}
\end{figure}

\vspace{-0.2cm}
\paragraph{Related work.} While Chebyshev representations of $f$ are often considered when $m = 1$, their use in computational algebraic geometry $(m > 1)$ has remained limited. Exceptions include \cite{nakatsukasa2015computing,sorber2014numerical}, which use B\'ezout/Sylvester resultants in the Chebyshev basis for $m = 2$. The recent articles~\cite{mourrain2021truncated,parkinson2024chebyshev,parkinson2022analysis} consider general $m$. In particular, the experiments in \cite[Section 6.7]{mourrain2021truncated} show that more accurate results are obtained using Chebyshev representations, compared with monomials. Chebyshev curves, i.e., one-dimensional Chebyshev varieties, were studied in \cite{freudenburg2009curves}. To the best of our knowledge, more general varieties parametrized by Chebyshev polynomials and their multivariate analogs have not been studied before.

\vspace{-0.2cm}
\paragraph{Contributions and outline.} This paper contains several results related to dimension, degree, singular locus and defining equations of Chebyshev varieties. Section \ref{sec:2} explains the toric/monomial setting as a motivation. It makes the analogy between toric varieties and Chebyshev varieties explicit via sparse systems of equations. Section \ref{sec:3} is about curves $(m = 1)$ parametrized by Chebyshev polynomials of the first kind ($T$-curves) or by Chebyshev polynomials of the second kind ($U$-curves). We show that the plane $T$-curve associated to the matrix $A = [a ~ a\!+\!1]$ is hyperbolic with respect to the origin for any $a$ (Theorem \ref{thm:hyperT}), and so is the corresponding $U$-curve (Theorem \ref{thm:hyperU}). We derive generators for the ideal of a $T$-curve, under some assumptions on $A$ (Theorem \ref{thm:maincurves}). In Section \ref{sec:4}, we investigate the dimension and degree of Chebyshev varieties parametrized by tensor product Chebyshev polynomials (Proposition \ref{prop:degtensor} and Theorems \ref{thm:dimtensor}, \ref{thm:betterbound}). 
In this section, the polynomials $f_i$ in \eqref{eq:fi} could be multivariate Chebyshev interpolants approximating transcendental functions $\phi_i$, as in \cite{trefethen2017cubature}. In Section \ref{sec:5}, we discuss dimension (Theorem \ref{thm:cosineChebDim}), degree (Theorem \ref{thm:degCos}), singular locus (Proposition \ref{prop:singCos}) and defining equations (Theorem \ref{thm:cosineChebDim}) for Chebyshev varieties parametrized by cosines of linear forms. 
Finally, in Section \ref{sec:6}, we discuss applications of our results and future research directions related to computing real roots and estimating their number. We also introduce Chebyshev varieties from root systems \cite{hubert2022sparse}. 
This paper aims to bring Chebyshev polynomials into the field of numerical/computational algebraic geometry. First algorithmic/experimental steps were taken in \cite{mourrain2021truncated,parkinson2024chebyshev,parkinson2022analysis}. Our hope is that Chebyshev varieties, like toric varieties, appeal to both theoretically and computationally inclined readers, and will lead to powerful numerical algorithms for solving equations.

\section{Toric varieties and sparse root finding} \label{sec:2}
For an algebraic geometer, solving sparse polynomial equations means intersecting toric varieties with linear spaces. Here \emph{sparse} means that only a prescribed set of monomials appears in each equation. This section aims to explain/justify this geometric interpretation, and to show how toric varieties are replaced by Chebyshev varieties when monomials are replaced by Chebyshev polynomials. Our prescribed monomials are encoded by $A \in \mathbb{N}^{m \times n}$: 
\begin{equation} \label{eq:fitoric}
 f_i(t) \, = c_{i,0} \, + \, c_{i,a_1} \, t^{a_1} \,  + \,  c_{i,a_2} \, t^{a_2} \, + \,  \cdots \, + \,  c_{i,a_n} \, t^{a_n}, \quad i = 1, \ldots, m.
\end{equation}
This uses the standard multi-index notation $t^{a_j} = t_1^{a_{1j}}t_2^{a_{2j}}\cdots t_m^{a_{mj}}$, where $a_j$ is the $j^{th}$ column of $A$. We aim to solve $f_1(t) = \cdots = f_m(t) = 0$, i.e., to compute $V_{\mathbb{C}^m}(f_1, \ldots, f_m)$. 

The \emph{affine toric variety} ${\cal Y}_A \subset \mathbb{C}^n$ associated to the matrix $A$ is parametrized by our monomials. I.e., it is the Zariski closure in $\mathbb{C}^n$ of the image of the following map: 
\[ \Phi_A \, : \, \mathbb{C}^m \longrightarrow \mathbb{C}^n, \quad  t \, \longmapsto \, (t^{a_1}, \ldots, t^{a_n}). \]
One checks that, when ${\rm rank}(A) = m$, the map $\Phi_A$ is generically $(\deg \Phi_A)$-to-one, meaning that almost all points in the image have $\deg \Phi_A$ pre-images. The number $\deg \Phi_A = [\mathbb{Z}^m:\mathbb{Z}A]$ is the index of the sublattice $\mathbb{Z} A \subset \mathbb{Z}^m$ generated by the columns of the matrix $A$.

Many geometric properties of ${\cal Y}_A$ are encoded by a convex polytope related to $A$. We take the convex hull of the columns of $A$, viewed as points in $\mathbb{R}^m$, and the origin in $\mathbb{R}^m$. The resulting polytope is denoted by $P_A = {\rm conv}(A \cup 0) \subset \mathbb{R}^m$.\footnote{For the expert reader, we point out that this is slightly different from the usual setup in toric geometry \cite{cox2011toric,telen2022introduction}. One typically allows negative entries in the matrix $A$, but we prefer to use polynomials. Appending the origin $0$ to the matrix $A$ corresponds to including constant terms $c_{i,0}$ in \eqref{eq:fitoric}. We are implicitly viewing ${\cal Y}_A$ as an affine chart of the projective toric variety associated to $P_A$. } The following theorem states how to read off the dimension, degree and defining equations of ${\cal Y}_A$ from $A$ and $P_A$.  
 
\begin{theorem} \label{thm:maintoric}
    The dimension of the affine toric variety ${\cal Y}_A$ equals the rank of the matrix $A$, which equals the dimension of the polytope $P_A = {\rm conv}(A \cup 0)$. If ${\rm rank}(A) = m$, we have 
    \begin{equation} \label{eq:degtoric}
    \deg {\cal Y}_A \, = \, \frac{m!}{\deg \Phi_A} \cdot {\rm vol}(P_A).
    \end{equation}
    The ideal of ${\cal Y}_A$ is generated by the binomials $x^u - x^v$, for $u, v \in \mathbb{N}^n$ and $u-v \in \ker_{\mathbb{Z}} A$. 
\end{theorem}
This theorem summarizes well-known results in toric geometry, see for instance \cite[Corollary 2.14 and Theorems 2.19, 3.16]{telen2022introduction}. It uses the notation $\ker_{\mathbb{Z}}(A)$ for all vectors in the kernel of $A$ with integer entries. Readers who are unfamiliar with the concept of degree of an algebraic variety can use the following definition of $\deg {\cal Y}_A$: a generic affine-linear space of dimension $n-m$ intersects ${\cal Y}_A$ in $(\deg {\cal Y}_A)$-many points. We also recall that the (vanishing) ideal of ${\cal Y}_A$ is the ideal of $\mathbb{C}[x_1, \ldots, x_n]$ containing all polynomials vanishing on ${\cal Y}_A$.

\begin{example}
    The $1 \times 2$ matrix $A = [5~7]$ defines a toric curve ${\cal Y}_A \subset \mathbb{C}^2$ parametrized by $\Phi_A(t) = (t^5,t^7)$. 
    The map $\Phi_A$ is one-to-one, which agrees with $\mathbb{Z}A = \mathbb{Z}$. The polytope $P_A$ is the line segment $[0,7]$, with volume 7. This is the degree of our curve. We have ${\cal Y}_A = \{(x,y) \in \mathbb{C}^2\,:\, x^7-y^5 = 0\}$. This binomial equation corresponds to $[7~0]^\top - [0~5]^\top\in \ker_{\mathbb{Z}} A$.  
    Notice that the matrix $A' = [10~14]$ parametrizes the same curve ${\cal Y}_{A'} = {\cal Y}_A$, but $\Phi_{A'}$ is 2-to-1. The subgroup $\mathbb{Z}A \subset \mathbb{Z}$ is $2 \mathbb{Z}$. The number in \eqref{eq:degtoric} is still 7 after replacing $A$ by $A'$. 
\end{example}

\begin{example} \label{ex:pringles1}
    We consider the toric surface ${\cal Y}_A  \subset \mathbb{C}^3$ coming from the matrix 
    \[ A \, = \, \begin{bmatrix}
        1 & 1 & 2 \\ 2 & 1 & 3
    \end{bmatrix}.\]
    Since ${\rm rank}(A) = 2$, this is indeed a surface, see Figure \ref{fig:sub-first}. The polytope $P_A$ is the quadrilateral in the left part of Figure \ref{fig:degpringles}. The map is $\Phi_A(t_1,t_2) = (t_1t_2^2, t_1 t_2, t_1^2t_2^3)$. The Zariski closure of its image is ${\cal Y}_A = \{(x,y,z) \in \mathbb{C}^3 \, : \, xy-z = 0 \}$. Notice that $(1,0,0) \in {\cal Y}_A$, but it is not in ${\rm im} \, \Phi_A$: taking the closure is necessary to obtain a variety. The degree of the equation $xy-z$ is 2, which is $2!$ times the volume of $P_A$, as predicted by Theorem \ref{thm:maintoric}. 
    \begin{figure}[ht]
  \centering
  \begin{tikzpicture}[scale=0.83]
\begin{axis}[%
width=1.2in,
height=1.8in,
scale only axis,
xmin=-0.5,
xmax=2.5,
ymin=-0.5,
ymax=3.5,
ticks = none, 
ticks = none,
axis background/.style={fill=white},
axis line style={draw=none} 
]

\draw[->,>=stealth] (axis cs:0,0) -- (axis cs:2.5,0);
\draw[->,>=stealth] (axis cs:0,0) -- (axis cs:0,3.5);

\addplot[only marks,mark=*,mark size=1.1pt,black
        ]  coordinates {
   (0,0) (1,0) (2,0) (0,1) (1,1) (2,1) (0,2) (1,2) (2,2) (0,3) (1,3) (2,3)
};

\addplot [color=myblue,solid,fill opacity=0.2,fill = myblue,forget plot]
  table[row sep=crcr]{%
 0 0\\
1 2\\	
2 3\\
1 1 \\
0 0\\
};

\addplot [very thick, color=myblue,solid,fill opacity=0.2,fill = myblue,forget plot]
  table[row sep=crcr]{%
 0 0 \\
 1 2\\
};

\addplot [very thick, color=myblue,solid,fill opacity=0.2,fill = myblue,forget plot]
  table[row sep=crcr]{%
 1 2 \\
 2 3 \\
};

\addplot [very thick, color=myblue,solid,fill opacity=0.2,fill = myblue,forget plot]
  table[row sep=crcr]{%
 2 3\\
 1 1 \\
};

\addplot [very thick, color=myblue,solid,fill opacity=0.2,fill = myblue,forget plot]
  table[row sep=crcr]{%
 0 0\\
 1 1 \\
};

\addplot[only marks,mark=*,mark size=3.1pt,myblue
        ]  coordinates {
  (0,0) (1,2) (1,1) (2,3)
};

\end{axis}
\end{tikzpicture}  
  \quad \qquad \quad 
  \includegraphics[height = 1.5in]{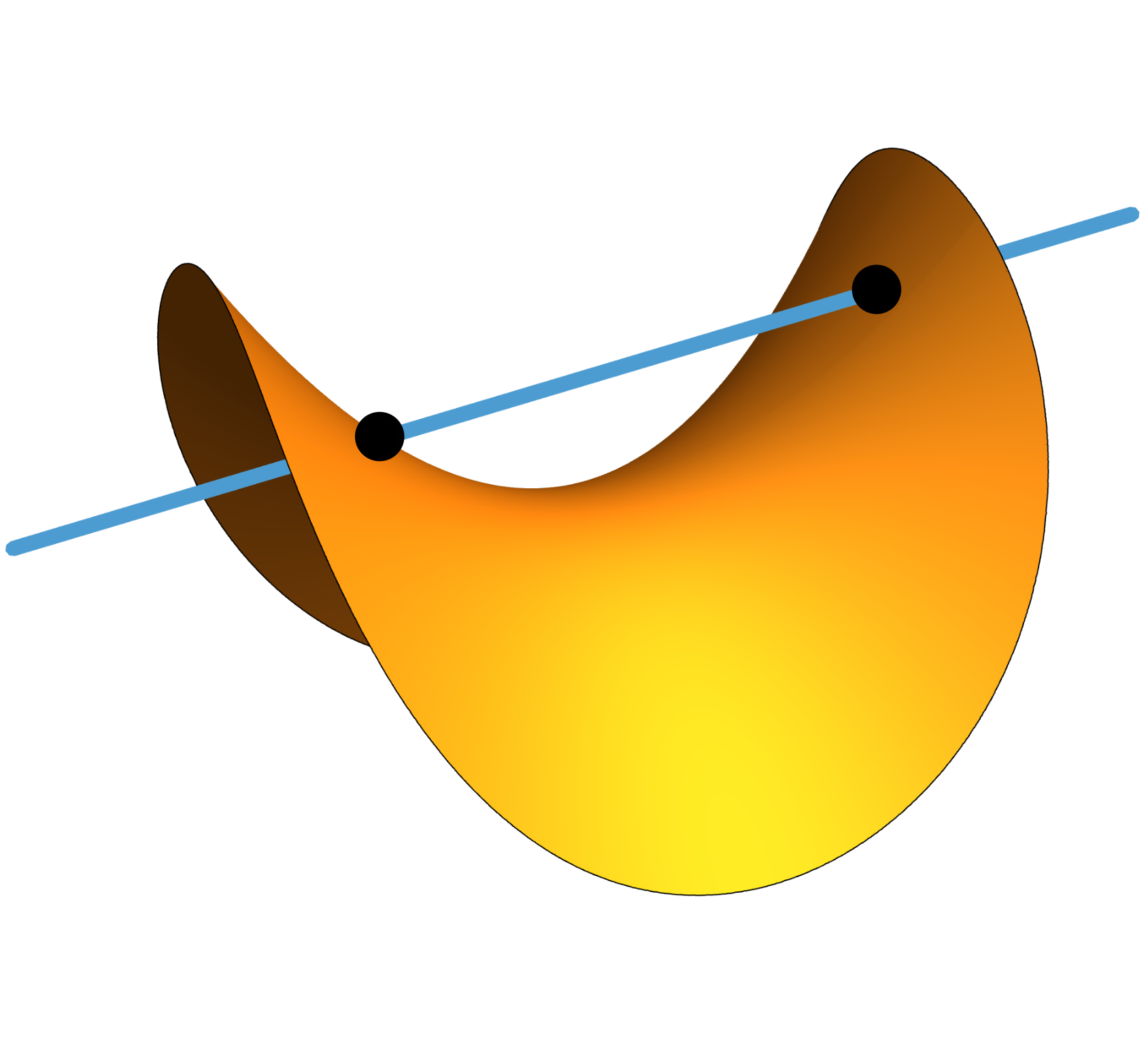}  
\caption{The number of intersection points of ${\cal Y}_A$ is $2!$ times the volume of $P_A$.}
\label{fig:degpringles}
\end{figure}
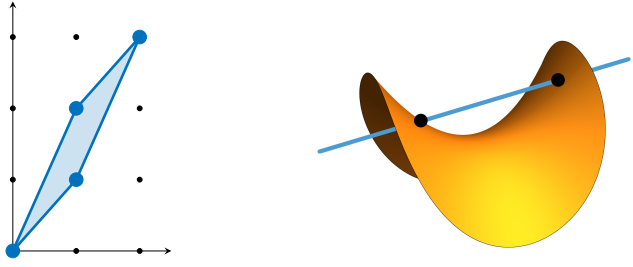
\end{example}

Let $c_0 \in \mathbb{R}^m$ and $C \in \mathbb{R}^{m \times n}$ be as in the introduction. The relevance of Theorem \ref{thm:maintoric} for the problem of solving $f_1 = \cdots = f_m = 0$, with $f_i$ from \eqref{eq:fitoric}, is seen from the easy implication
\[ t \, \in \,  V_{\mathbb{C}}(f_1, \ldots, f_m) \, \, \Longrightarrow \, \, c_0 + C \cdot \Phi_A(t) \, = \, 0. \]
Conversely, if $x \in {\rm im} \, \Phi_A \subset \mathbb{C}^n$ and $c_0 + C \cdot x = 0$, then any preimage $t \in \Phi_A^{-1}(x)$ lies in $V_{\mathbb{C}}(f_1, \ldots, f_m)$. It is customary to relax the condition $x \in {\rm im} \, \Phi_A$ slightly by allowing points in the closure: $x \in {\cal Y}_A$. This breaks up our problem into two steps: 
\begin{enumerate}
    \item[(1)] Compute all intersections of ${\cal Y}_A$ with the affine-linear space $\{ x \in \mathbb{C}^n \, : \, c_0 + C \cdot x = 0 \}$.
    \item[(2)] Compute $\Phi_A^{-1}(x)$ for each point $x$ in the above intersection which lies in ${\rm im} \, \Phi_A$. 
\end{enumerate}
First of all, if $\dim {\cal Y}_A < m$, the intersection in step (1) is expected to be empty. That is, if the entries of $c_0, C$ are generic, the affine-linear space given by $c_0 + C \cdot x = 0$ does not intersect ${\cal Y}_A$. When $\dim {\cal Y}_A = m$, the number of points obtained in step (1) is $\deg {\cal Y}_A$. In step (2) each of these $(\deg {\cal Y}_A)$-many points gives $\deg \Phi_A$ points in the fiber $\Phi_A^{-1}(x)$. These points are easy to compute, so most of the work is in step (1). If the affine-linear space $c_0 + C \cdot x = 0$ is tangent to ${\cal Y}_A$ at some point, or it hits the singular locus ${\rm sing}({\cal Y}_A)$, we find a singular solution (i.e., a solution with multiplicity $>1$) of $f_1 = \cdots = f_m = 0$.
\begin{example} \label{ex:pringles2}
    The matrix $A$ from Example \ref{ex:pringles1} corresponds to the $m = 2$ equations 
    \begin{equation} \label{eq:pringleseq} c_{1,0} \, + \, c_{1,1} \, t_1t_2^2 \, + \, c_{1,2} \, t_1t_2 \, + \, c_{1,3} \, t_1^2t_2^3 \, = \, c_{2,0} \, + \, c_{2,1} \, t_1t_2^2 \, + \, c_{2,2} \, t_1t_2 \, + \, c_{2,3} \, t_1^2t_2^3 \, = \, 0. \end{equation}
    Step (1) is to intersect the toric surface ${\cal Y}_A$ with the line defined by $c_0 + C \cdot x = 0$, see the right part of Figure \ref{fig:degpringles}. The generic number of intersection points is $\deg {\cal Y}_A = 2$. Since $\deg \Phi_A = 1$, the two intersection points are in bijection with the two solutions of \eqref{eq:pringleseq}.
\end{example}
The above discussion justifies the statement that the affine toric variety ${\cal Y}_A$ governs the geometry of solving $f_1 = \cdots = f_m = 0$ when $f_i$ is as in \eqref{eq:fitoric} and the coefficients $c_0, C$ vary. The goal of our paper is to replace the monomials $t^{a_j}$ in \eqref{eq:fitoric} by (generalizations of) the Chebyshev polynomials $T_k(x)$, and prove results similar to Theorem \ref{thm:maintoric} in this new setting. 

\section{Chebyshev curves} \label{sec:3}

The first case to consider is $m = 1$. This corresponds to one-dimensional Chebyshev varieties, i.e., \emph{Chebyshev curves}. We start with plane Chebyshev curves in Section \ref{subsec:planecurves}, which corresponds to $n = 2$. The entries of the matrix $A \in \mathbb{N}^{1 \times 2}$ are denoted by $A = [a~b]$. We consider \emph{$T$-curves} $(T_a(t), T_b(t))$ parametrized by the Chebyshev polynomials of the first kind. Similarly, \emph{$U$-curves} $(U_a(t), U_b(t))$ are parametrized by second kind Chebyshev polynomials. In Section \ref{subsec:spacecurves}, we study Chebyshev curves embedded in ambient spaces of higher dimension. 

\subsection{Plane curves} \label{subsec:planecurves}
Let $A=[a~b] \in \mathbb{N}^{1 \times 2}$ and consider the curve parametrized by the map $\mathcal{T}_{a,b}: \mathbb{C} \rightarrow \mathbb{C}^2$, with 
\[
\mathcal{T}_{a,b}(t)\, = \, (T_a(t), T_b(t)).
\]
The Zariski closure of the image is denoted by $\mathcal{X}_{a,b,T}$. Since the $k^{th}$-Chebyshev polynomial is the unique polynomial satisfying $T_k(\cos \theta) = \cos(k \, \theta)$, $\mathcal{X}_{a,b,T}$ is alternatively parametrized by $(\cos(a \, \theta),\cos(b \, \theta))$. Its real part is an instance of a \emph{Lissajous curve} arising in oscillator physics, see \cite{greenslade1993all}. These are more generally given by $x=\alpha \, \cos(a \,  \theta)$, $y=\beta \, \cos(b \, \theta + \phi)$. Here, we stick to $\alpha=\beta=1, \phi=0$. The implicit equation of $\mathcal{X}_{a,b,T}$ is easy to find \cite{freudenburg2009curves}. 
\begin{theorem}
    The ideal of $\mathcal{X}_{a,b,T}$ is generated by the irreducible polynomial \[ T_{b'}(x) - T_{a'}(y), \quad \text{where} \quad b'=\frac{b}{\textup{gcd}(a,b)}, \, a' = \frac{a}{\textup{gcd}(a,b)}.  \]
\end{theorem}

\begin{proof}
The Chebyshev $T$-polynomials satisfy the following equalities:
\[ T_{n}(T_{m}) = T_{m}(T_{n})=T_{nm}(t), \quad \text{and thus } \quad T_{a}(t) = T_{a'}(T_{{\rm gcd}(a,b)}(t)).\]
Hence the Chebyshev curve $\mathcal X_{a,b,T}$ is contained in the variety defined by $T_{b'}(x) - T_{a'}(y) = 0$. The claim follows by proving that this polynomial is irreducible, see \cite[Section 3]{freudenburg2009curves}.
\end{proof} 

\begin{corollary}
    Let $a,b \in \mathbb{N}$ be such that $0<a\leq b$. The degree of $\mathcal{X}_{a,b,T}$ is $b'=\frac{b}{{\rm gcd}(a,b)}$. 
\end{corollary}

The singularities of $\mathcal{X}_{a,b,T}$ can be found from the defining equation \cite[Proposition 3.1]{freudenburg2009curves}. 
\begin{theorem} \label{thm:singTcurve}
    The singular points $\textup{Sing}(\mathcal{X}_{a,b,T})$ of the Chebyshev curve $\mathcal{X}_{a,b,T}$ are 
    \[ 
    \textup{Sing}(\mathcal{X}_{a,b,T})=\{(\cos(\tfrac{k\pi}{b}), \cos(\tfrac{l\pi}{a})) : 1\leq k \leq b-1, \quad 1 \leq l \leq a-1, \quad k \equiv l \mod{2}\}.
    \]
    Moreover, all these points are nodal singularities of $\mathcal{X}_{a,b,T}$.
\end{theorem}
The points $\textup{Sing}(\mathcal{X}_{a,b,T})$ are used in bivariate polynomial interpolation. In that context, they are called \emph{Padua points}, see for instance \cite{bos2006bivariate,erb2016bivariate}.
The roots of the Chebyshev polynomial $T_a(t)$ are $\cos((1/2+l)\pi/a), l = 0, \ldots a-1$. We use this simple expression to count real intersections of ${\cal X}_{a,b,T}$ with lines through the origin.

\begin{theorem}\label{thm:hyperchev}
    Consider $ 0 < a < b $ in $\mathbb{N}$, so that $a,b$ are coprime. Any line passing through the origin intersects the curve $\mathcal{X}_{a,b,T}$ in at least $a$ real points, counting multiplicities.
\end{theorem}

\begin{proof}
    We prove that, for generic $(\alpha, \beta) \in \mathbb{S}^1$, the polynomial $\alpha T_{a}(t) - \beta T_{b}(t) \in \mathbb{R}[t]$ has at least $a$ distinct real roots. Since ${\rm gcd}(a,b)=1$, ${\cal T}_{a,b}$ is generically one-to-one. The roots of $\alpha T_{a}(t) - \beta T_{b}(t)$ will then correspond to distinct real intersections of $\mathcal{X}_{a,b,T}$ with the line $\alpha x- \beta y=0$. We denote by $x_k = \cos(\frac{(k+1/2)\pi}{a})$, resp. $y_l = \cos(\frac{(l+1/2)\pi}{b})$, the roots of $T_a$, resp. $T_b$. Here $k$ ranges from $0$ to $a-1$, and $l=0,\ldots,b-1$. We prove that between two  consecutive roots of $T_a$ there is a root of $T_b$. Indeed, we can find $l \in \{0,\ldots, b-1\}$ such that 
    $$
    \cos\frac{(k+3/2)\pi}{a} \, < \,  \cos\frac{(l+1/2)\pi}{b} \, < \, \cos\frac{(k+1/2)\pi}{a}.
    $$
    This is equivalent to finding an integer in the interval $(\frac{b}{a}(k + \frac{1}{2}) - \frac{1}{2}, \frac{b}{a}(k+\frac{3}{2}) - \frac{1}{2}) = (y, y + \frac{b}{a})$. Since this interval has length $\frac{b}{a} > 1$, it contains an integer. Let $y_i$ be a root of $T_b$ such that $x_i < y_i < x_{i+1}$ for $i=1, \ldots a-1$. Note that we can also find two additional roots $y_0, y_{a}$ of $T_b$ such that $y_0 < x_1$ and $y_{a}>x_a$. Since the roots of $T_a$ are simple, $T_a$ changes sign between $y_i$ and $y_{i+1}$. By the intermediate value theorem, $\alpha T_{a}(t) - \beta T_{b}(t)$ has a root in each interval $(y_{i}, y_{i+1})$ for $i=0, \ldots, a-1$ and therefore it has at least $a$ distinct real roots.
\end{proof}

It follows from our proof of Theorem \ref{thm:hyperchev} that, like all orthogonal polynomials, consecutive Chebyshev polynomials $T_a,T_{a+1}$ have interlacing roots. 
In the special case of Chebyshev curves $\mathcal{X}_{a,a+1,T}$,  
Theorem \ref{thm:hyperchev} implies \textit{hyperbolicity}. 

\begin{theorem}\label{thm:hyperT}
For $a \in \mathbb{N} \setminus \{0\}$, the Chebyshev plane curve $\mathcal{X}_{a,a+1,T}$ is hyperbolic with respect to the origin. That is, any line passing through the origin intersects the curve $\mathcal{X}_{a,a+1,T}$ in $\deg(\mathcal{X}_{a,a+1,T})=a+1$ real points, counting multiplicities.
\end{theorem}

\begin{proof}
 By Theorem \ref{thm:hyperchev}, we know that for generic coefficients $(\alpha, \beta) \in \mathbb{S}^1$ the polynomial $\alpha T_{a}(t) - \beta T_{a+1}(t) \in \mathbb{R}[t]$ has at least $a$ distinct real roots. Since this polynomial has degree $a+1$, all of its $a+1$ roots are real. Thus, we have at least $a+1$ real intersection points with the line. Since $a+1$ is the degree of $\mathcal{X}_{a,a+1,T}$, this is the exact number of intersection points. A root has multiplicity 2 if the line $\alpha x- \beta y=0$ passes through a node of the curve.
\end{proof}

\begin{example}
    The Chebyshev curve $\mathcal{X}_{6,7,T}$ in Figure \ref{fig:chevhyper} has degree 7. Any line passing through $(0,0)$ intersects the curve in $\deg(\mathcal{X}_{6,7,T})=7$ real points, counting multiplicities.
\begin{figure}[ht]
    \begin{subfigure}{.5\textwidth}
    \centering
    \includegraphics[height=4.5cm]{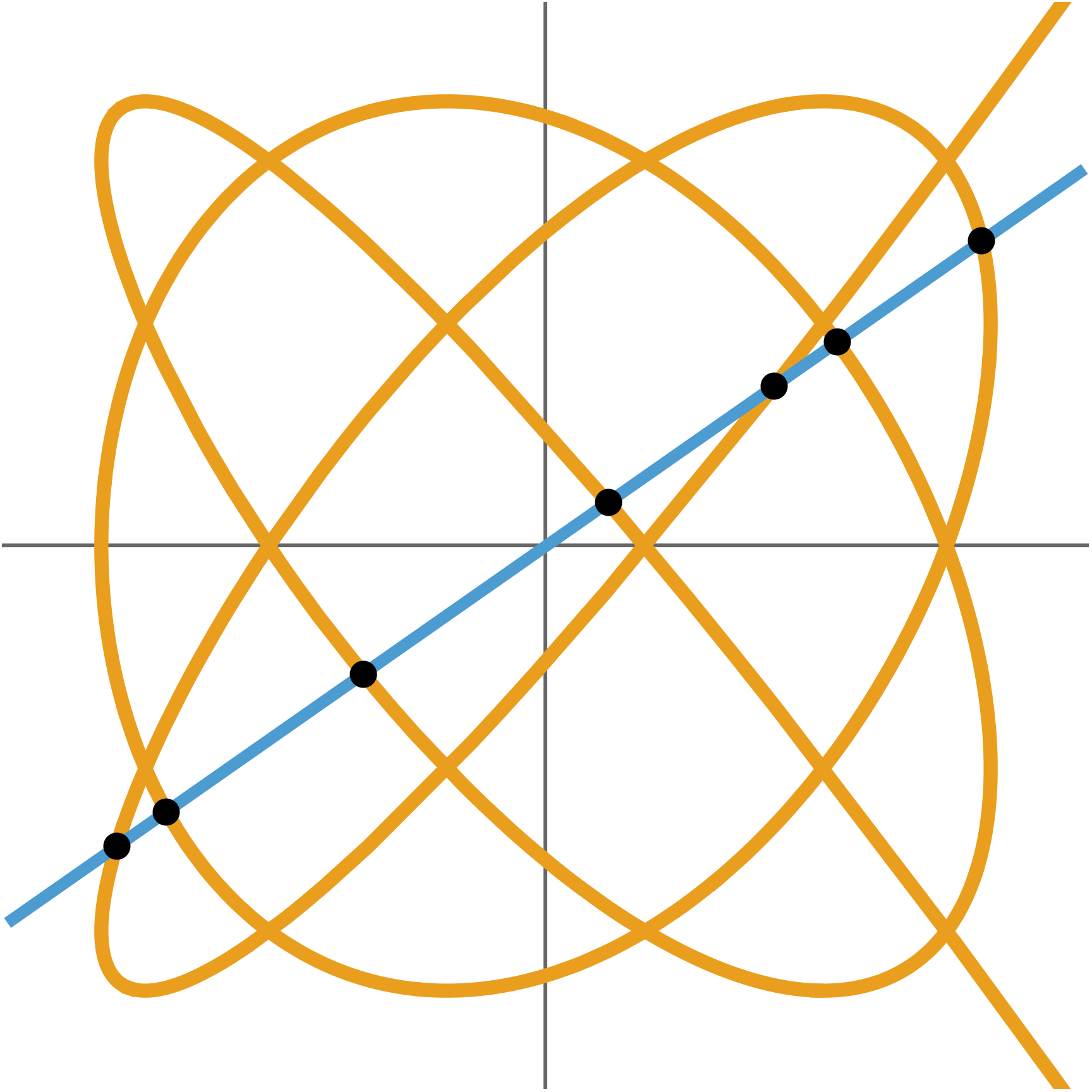}
    \caption*{Chebyshev curve $\mathcal{X}_{6,7,T}$}
    \end{subfigure}
    \begin{subfigure}{.5\textwidth}
    \centering
    \includegraphics[height=4.5cm]{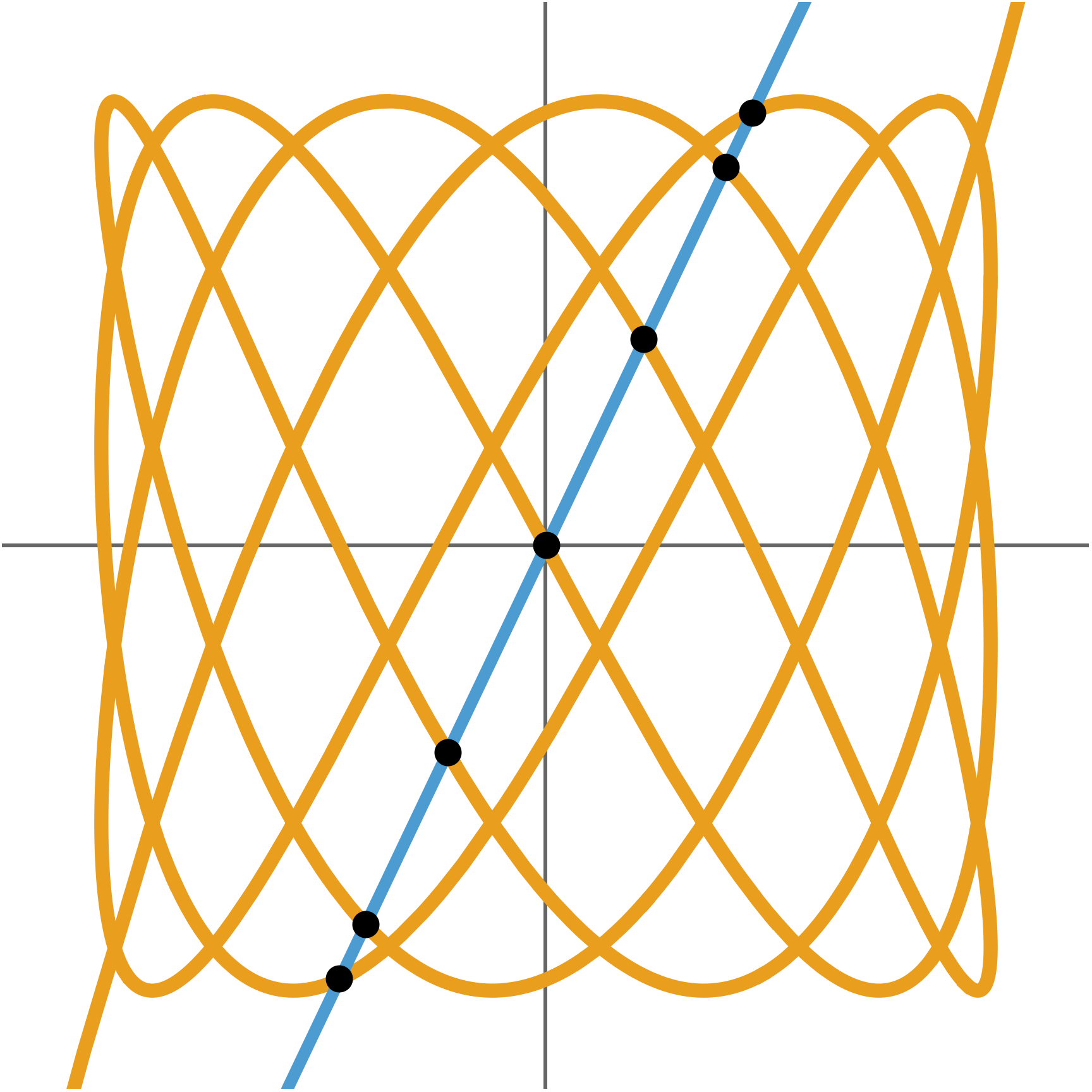}
    \caption*{Chebyshev curve $\mathcal{X}_{7,13,T}$}
    \end{subfigure}
    \caption{Chebyshev curves intersected by lines passing through the origin.}
\label{fig:chevhyper}    
\end{figure}
\end{example}

We close this subsection with a discussion on Chebyshev polynomials of the second kind:
\[ 
U_0(t) \, = \, 1, \quad U_1(t) \, = \, 2t, \quad \text{and} \quad U_{k+1}(t) \, = \, 2t \cdot U_{k}(t) - U_{k-1}(t). 
\]
The corresponding parametric plane curves are called \emph{Chebyshev $U$-curves}. We write $\mathcal{X}_{a, b, U}$ for the curve parametrized by $(U_a(t),U_b(t))$. Two examples are shown in Figure \ref{fig:chevu}.

In contrast with Chebyshev $T$-curves (Theorem \ref{thm:singTcurve}), Chebyshev $U$-curves may have singularities which are not simple nodes. For instance, the Chebyshev $U$-curve $\mathcal{X}_{5,7, U}$ has three identical tangent lines at the point $(-1,1)$ in black. The origin, also in black, has multiplicity 4 in the $U$-curve $\mathcal{X}_{9,14,U}$, see Figure \ref{fig:chevu}.

\begin{figure}[ht]
    \begin{subfigure}{.5\textwidth}
    \centering
    \includegraphics[height=4cm]{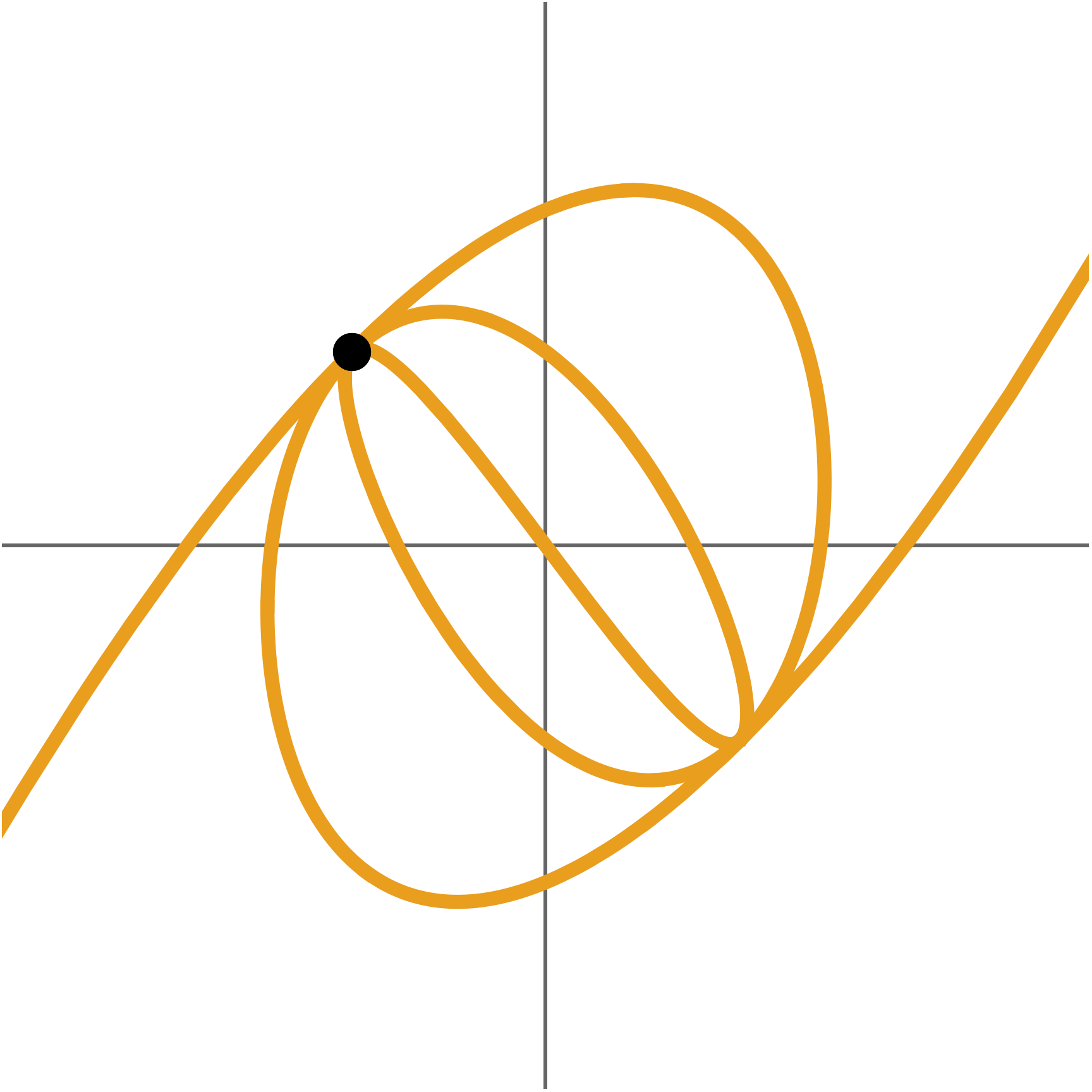}
    \caption*{Chebyshev $U$-curve $\mathcal{X}_{5,7,U}$}
    \end{subfigure}
    \begin{subfigure}{.5\textwidth}
    \centering
    \includegraphics[height=4cm]{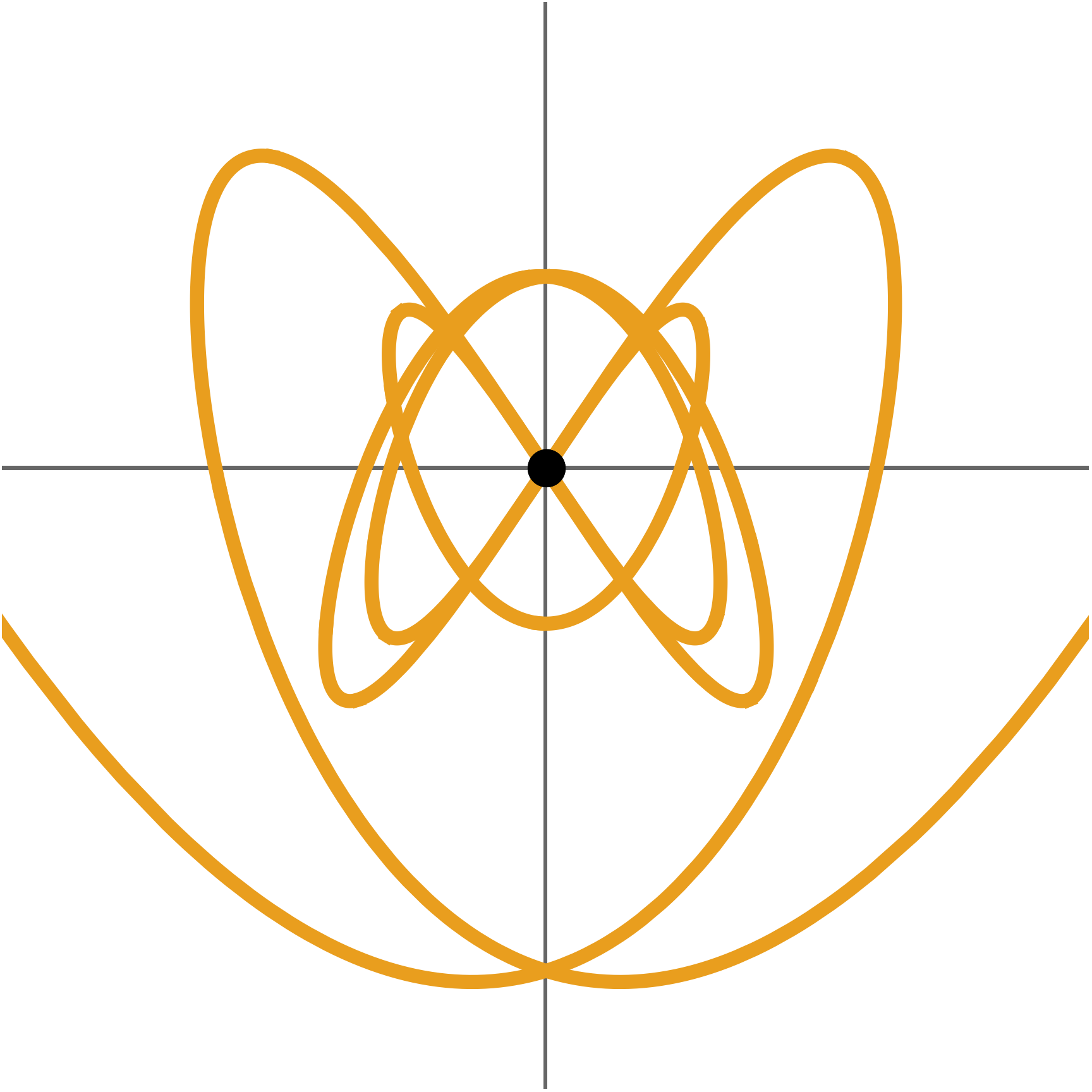}
    \caption*{Chebyshev $U$-curve $\mathcal{X}_{9,14,U}$}
    \end{subfigure}
    \caption{Chebyshev $U$-curves with different singularities.}
\label{fig:chevu}    
\end{figure}

We were not able to find the implicit equation of ${\cal X}_{a,b,U}$ in general. We illustrate by means of an example how such equations can be found using some algebraic properties of $U_k(t)$.

\begin{example} 
    We compute an implicit equation for ${\cal X}_{k,k+2,U}$. One checks that 
    \[ 
         2\, T_m(t) = U_m(t) - U_{m-2}(t)
, \quad \quad
    2\, T_m(t)U_n(t) = \left\{
    \begin{array}{ll}
        U_{m+n}(t)+ U_{n-m}(t) & \mbox{{\rm if}\, } n \geq m-1, \\
        U_{m+n}(t)- U_{n-m-2}(t) & \mbox{otherwise.}
    \end{array}
    \right.
    \]
    Combining these identities, and omitting $t$ to simplify the notation, we find that
$$
    (U_{k+2} - U_{k})U_k = U_{2k+2} -1, \quad (U_k - U_{k-2})U_{k+2} = U_{2k+2} + U_2.
$$
Using the fact that $U_{k+2} = U_k(U_2 - 1) - U_{k-2}$, we get
\[ (U_k - U_{k-2})U_{k+2} - (U_{k+2} - U_{k})U_k = \frac{U_{k+2} + U_{k-2}}{U_n} +2.\]
This gives an implicit equation for the parametric curve $(U_{k-2}(t), U_{k}(t), U_{k+2}(t))$ in $\mathbb{C}^3$:
\begin{equation} \label{eq:eqU}
    y^3  - 2y - z \, =\, x(1+yz).
\end{equation}
The composition property of Chebyshev $T$-polynomials gives a second relation:
\[ 
T_k\left (\frac{U_{k+2} - U_k}{2} \right ) = T_{k+2}\left (\frac{U_k - U_{k-2}}{2}\right ).
\]
We obtain an equation $f(y,z)=0$ for ${\cal X}_{k,k+2,U}$ as the numerator of $T_{k+2}(\frac{y-x}{2}) - T_k(\frac{z-y}{2}) = 0$ after substituting $x$ using \eqref{eq:eqU}. The resulting $f$ is reducible. For $k = 3$, we find 
\begin{align*}
    f(y,z)\, = \, &\left(y^8-3 y^7 z+3 y^6 z^2-7 y^6-y^5 z^3+12 y^5 z-5 y^4 z^2+15 y^4-10 y^3 z-10 y^2+1\right) \cdot\\
    &\left(y^5-4 y^3-3 y^2 z-2 y z^2+3 y-z^3+2 z\right) \cdot (y-z+2) \cdot  (y-z-2),
\end{align*}
where the degree 5 factor is the irreducible defining equation of ${\cal X}_{3,5,U}$.
\end{example}

Chebyshev $U$-polynomials are also orthogonal and $U_a,U_{a+1}$ have interlaced roots. This allows us to state the same hyperbolicity property as in Theorem \ref{thm:hyperT}. 

\begin{theorem}\label{thm:hyperU}
For $a \in \mathbb{N} \setminus \{0\}$, the Chebyshev plane curve $\mathcal{X}_{a,a+1,U}$ is hyperbolic with respect to the origin, in the sense that all intersection points of the curve with any line passing through the origin are real. 
\end{theorem}

\subsection{Space curves} \label{subsec:spacecurves}
We now consider Chebyshev space curves parametrized by  
\[ 
\mathcal{T}_{A,T}:\mathbb{C} \rightarrow \mathbb{C}^n, \quad \mathcal{T}_{A,T}(t)=(T_{a_1}(t),\ldots, T_{a_n}(t)),
\]
where $A=[a_1~\cdots~a_n] \in \mathbb{N}^{1 \times n}$. We denote the Zariski closure of the image by $\mathcal{X}_{A,T}$.
\begin{theorem} \label{thm:maincurves}
Let $A = [a_1~\cdots~a_n] \in \mathbb{N}^{1 \times n}$ be such that $0 < a_1 \leq \cdots \leq a_n$. If three entries $a_i, a_j, a_k \in A$ are pairwise coprime, then there exists a polynomial $P \in \mathbb{R}[x_i,x_j,x_k]$ such that $P(T_{a_i}(t),  T_{a_j}(t), T_{a_k}(t)) = t$. Hence, the Chebyshev curve $\mathcal{X}_{A,T}$ is smooth. Its prime ideal~is
\[ I(\mathcal{X}_{A,T}) \, = \, \langle \,  x_1 - T_{a_1}(P(x_i,x_j,x_k)), \, \ldots, \, x_n - T_{a_n}(P(x_i, x_j, x_k)) \, \rangle \, \subset \, \mathbb{C}[x_1, \ldots, x_n]. \]
\end{theorem}
Notice that, while Chebyshev plane curves are singular (Theorem \ref{thm:singTcurve}), the space curve ${\cal X}_{A,T}$ is smooth under mild conditions on $A$. The key lemma here is \cite[Proposition~4.1]{freudenburg2009curves}.
\begin{lemma}\label{lem:equaring}  Let $1 <a\leq b \leq c$. The inclusion $\mathbb{C}[T_a(t), T_b(t), T_c(t)] \subseteq \mathbb{C}[t]$ is an equality if and only if $a, b, c$ are pairwise coprime.
\end{lemma}
While a complete proof of the statement can be found in \cite{freudenburg2009curves}, it is instructive to explicitly construct a polynomial $P \in \mathbb{R}[x,y,z]$ such that $P(T_a(t), T_b(t),T_c(t))=t$ in the case where $a$, $b$ and $c$ are pairwise coprime. At most one of $a$, $b$, $c$ is even. Without loss of generality, we assume that $a$ is odd. Since $b$ and $c$ are coprime, one can find positive integers $u, v$ such that $ub - vc = 1$. Using standard properties of Chebyshev polynomials, we have
\[ 2T_{ub}T_{vc} = T_{ub + vc} + T_{ub - vc}, \quad \text{and thus} \quad 2T_u(T_{b})T_v(T_{c}) = T_{ub + vc} + T_1.\]
It suffices to prove that we can find $u$ and $v$ such that $ub-vc = 1$ and $ub + vc = wa$ for some $w \in \mathbb{Z}_{\geq 0}$. Indeed, then the equality above reads $2T_u(T_b)T_v(T_c) - T_w(T_a) = t$, and $P = 2T_u(y)T_v(z)-T_w(x)$. The equality $ub - vc = 1$ still holds after replacing $u$ by $u + xc$ and $v$ by $v + xb$, where $x$ is a positive integer. Therefore, one can take $x$ such that $2xbc + (ub + vc) \equiv 0 \mod{a}$. Such an $x$ exists, because $a$ is odd and thus coprime with $2bc$.

\begin{lemma} \label{lem:primeideal} 
Let $f = (f_1, \ldots, f_n) : \mathbb{C} \rightarrow \mathbb{C}^n$ be a polynomial map, i.e., $f(t) \in \mathbb{C}[t]$. Let $X \subset \mathbb{C}^n$ be the Zariski closure of ${\rm im} \, f$. If $\mathbb{C}[f_1, \ldots, f_n] = \mathbb{C}[t]$ and $P \in \mathbb{C}[x_1, \ldots, x_n]$ is such that $P(f_1, \ldots, f_n) = t$, then $X \subset \mathbb{C}^n$ is a smooth curve with prime ideal 
\begin{equation} \label{eq:J}
I(X) \, = \, \langle x_1 - f_1(P(x)), \ldots, x_n - f_n(P(x)) \rangle. 
\end{equation}
\end{lemma}
\begin{proof}
Let $J$ be the right-hand side of \eqref{eq:J}. We claim that there is a short exact sequence
\[ 0 \longrightarrow J \longrightarrow \mathbb{C}[x_1, \ldots, x_n] \overset{\varphi}{\longrightarrow} \mathbb{C}[t] \longrightarrow 0 ,\]
where the $\mathbb{C}$-algebra homomorphism $\varphi: \mathbb{C}[x_1, \ldots, x_n] \longrightarrow \mathbb{C}[t]$ sends $x_i$ to $f_i(t)$. This is surjective because $P$ maps to $t$: $\varphi(P) = P(f_1(t), \ldots, f_n(t)) = t$. It remains to show that $\ker \varphi = J$. The inclusion $J \subset \ker \varphi$ is easily checked. For the other inclusion, note that for any polynomial $h \in \mathbb{C}[x_1, \ldots, x_n]$, we have $h - h(f_1(P), \ldots, f_n(P)) \in J$. This shows the desired implication $h \in \ker \varphi \Rightarrow h \in J$. We conclude that $\mathbb{C}[x_1, \ldots, x_n]/J \simeq \mathbb{C}[P]$ is an integral domain, so that $J$ is a one-dimensional prime ideal contained in $I(X)$. Since $I(X)$ is one-dimensional and prime as well, we have $J = I(X)$. 
\end{proof}

\begin{proof}[Proof of Theorem \ref{thm:maincurves}]
The existence of $P$ follows from Lemma \ref{lem:equaring}. The generators for $I(\mathcal{X}_{A,T})$ are found from Lemma \ref{lem:primeideal}. 
\end{proof}

\begin{corollary} 
If three entries of $A \in \mathbb{N}^{1 \times n}$ are pairwise coprime and $0 < a_1 \leq \cdots \leq a_n$, then the Chebyshev curve $\mathcal{X}_{A,T} \subset \mathbb{C}^n$ has degree $a_n$. 
\end{corollary}
\begin{proof}
The degree of $\mathcal{X}_{A,T}$ is the number of points in its intersection with a generic affine hyperplane $H = \{c_0 + c_1 x_1 + \cdots + c_n x_n = 0 \}$. Since $P(T_{a_i}(t),  T_{a_j}(t), T_{a_k}(t)) = t$, the map ${\cal T}_{A}: \mathbb{C} \rightarrow \mathbb{C}^n$ is a closed embedding. Hence, the intersection points $H \cap \mathcal{X}_{A,T}$ are in one-to-one correspondence with the complex roots of $c_0 + c_1 T_{a_1}(t) + \cdots + c_n T_{a_n}(t) = 0$. Since the coefficients $c_i$ are generic, there are $a_n$ such roots. 
\end{proof}

\begin{example}
    Let $A = [3~2~7]$ and consider the Chebyshev space curve ${\cal X}_{A,T}$ displayed in Figure \ref{fig:space_curve}, left.
    \begin{figure}[ht]
        \begin{subfigure}{.5\textwidth}
        \centering
        \includegraphics[height=4cm]{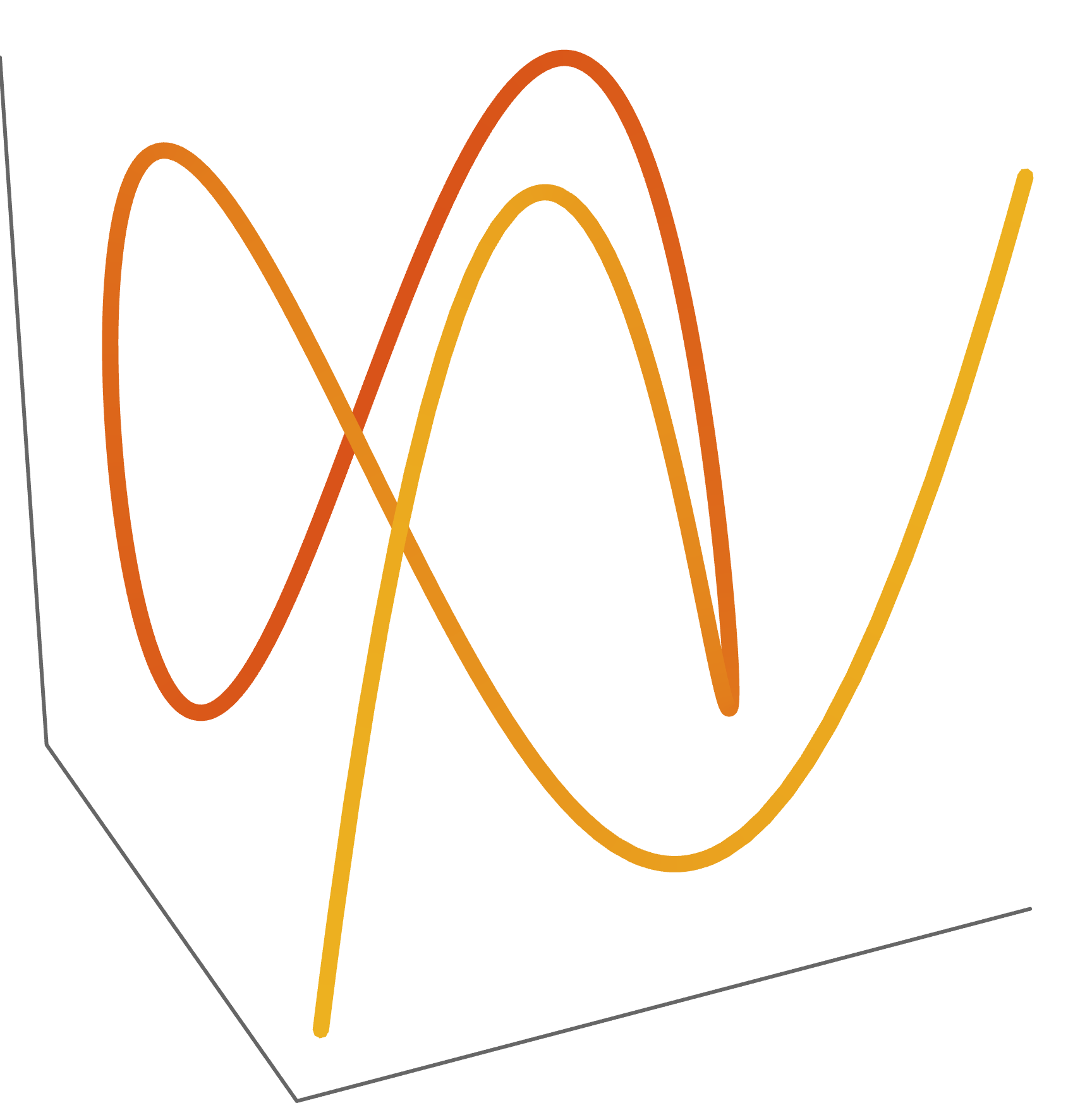}
        \caption*{Chebyshev $T$-curve $\mathcal{X}_{2,3,7,T}$}
        \end{subfigure}
        \begin{subfigure}{.5\textwidth}
        \centering
        \includegraphics[height=4cm]{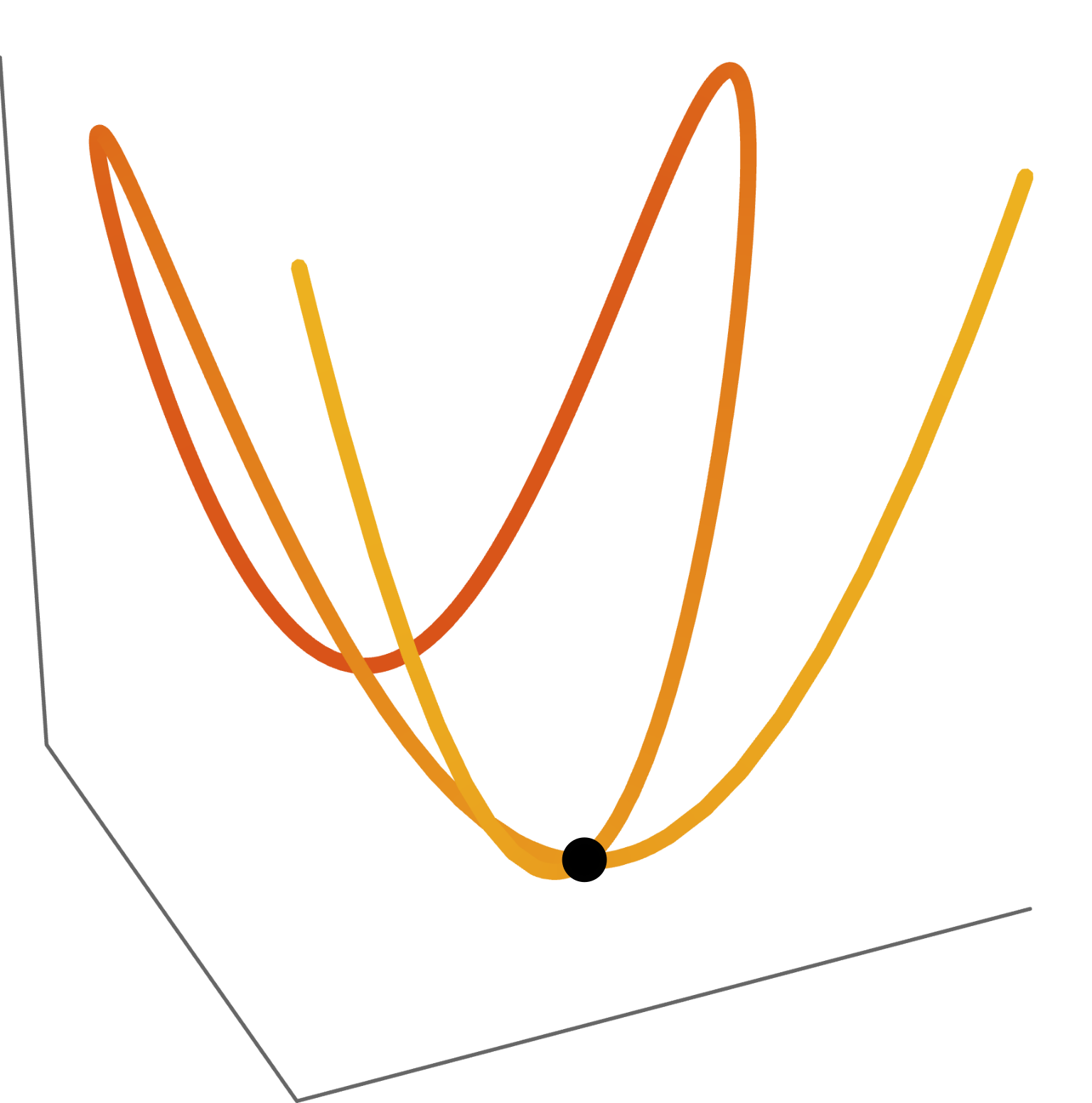}
        \caption*{Chebyshev $T$-curve $\mathcal{X}_{2,3,6,T}$}
        \end{subfigure}
        \caption{A smooth (left) and a singular (right) Chebyshev space curve.}
    \label{fig:space_curve}    
    \end{figure}
    For $P\in \mathbb{C}[x,y,z]$ from Lemma \ref{lem:primeideal} we can choose
    \[
    P(x,y,z) = 2 T_4(y) T_1(z) - T_5(x) = -16 x^5+16 y^4 z+20 x^3-16 y^2 z-5 x+2z.
    \]
    This polynomial satisfies $P(T_3(t),T_2(t),T_7(t)) = t$. Several such polynomials exist. E.g., 
    \[
    \widetilde{P}(x,y,z) = 2 T_{25}(y) T_7(z) - T_{33}(x) = -4294967296 x^{33} +2147483648 y^{25} z^7 + \ldots -33 x
    \]
    also satisfies $\tilde{P}(T_3(t),T_2(t),T_7(t)) = t$. This gives different generators for the same ideal.
\end{example}

\section{Tensor product Chebyshev polynomials} \label{sec:4}

We now switch to Chebyshev varieties of dimension $>1$. These correspond to polynomial equations $f_1 = \cdots = f_m = 0$, where each $f_i$ is a polynomial in $m > 1$ variables of the form \eqref{eq:fi}. For $a_j = (a_{1j}, \ldots, a_{mj}) \in \mathbb{N}^m$, we let ${\cal T}_{a_j,\otimes}$ be the \emph{tensor product Chebyshev polynomial}
\[ {\cal T}_{a_j, \otimes}(t) \, = \, T_{a_{1j}}(t_1) \cdot T_{a_{2j}}(t_2) \cdot \ldots \cdot T_{a_{mj}}(t_m). \]
This multivariate generalization of the Chebyshev polynomials $T_k$ has been used in numerical analysis, see for instance \cite{nakatsukasa2015computing,trefethen2017cubature}. This section studies the corresponding parametric varieties. 
The matrix $A = [a_1~\cdots~a_n]\in \mathbb{N}^{m\times n}$ defines the parametrization ${\cal T}_{A,\otimes}: \mathbb{C}^m \rightarrow \mathbb{C}^n$, with 
\begin{equation} \label{eq:tensormap} {\cal T}_{A,\otimes}(t) \, = \, ( \, {\cal T}_{a_1,\otimes}(t), {\cal T}_{a_2,\otimes}(t), \ldots, {\cal T}_{a_n,\otimes}(t) \, ).\end{equation}
The Zariski closure of its image is the \emph{(tensor product) Chebyshev variety} ${\cal X}_{A,\otimes}$. 
\begin{remark} \label{rem:densechoices}
    When $f_i$ is a Chebyshev approximation of some more complicated function $\phi_i$, the columns of the matrix $A$ often consist of all exponents of a fixed degree. Here ``degree'' can have different meanings. We present three examples following \cite[Section 6]{trefethen2017cubature}:
    \begin{itemize}
        \item \emph{Total} degree $k$: $A$ has columns $\{ a_j \in \mathbb{N}^m \, : \, \sum_{i=1}^m a_{ij} \leq k\}$. 
        \item \emph{Max} degree $k$: $A$ has columns $\{a_j \in \mathbb{N}^m \, : \, a_{ij} \leq k, \, i = 1, \ldots, m \}$. 
        \item \emph{Euclidean} degree $k$: $A$ has columns $\{a_j \in \mathbb{N}^m \, : \, \sum_{i=1}^m a_{ij}^2 \leq k^2 \}$.
    \end{itemize}
    This is illustrated in Figure \ref{fig:degrees}. 
    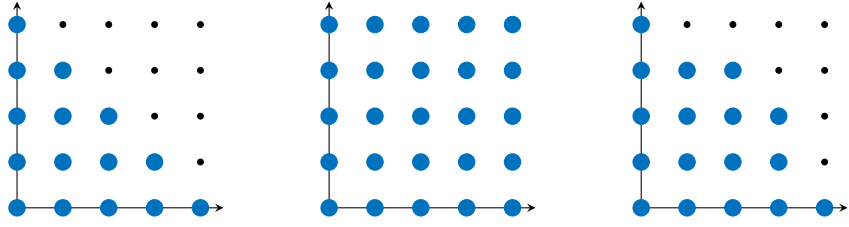
\begin{figure}
        \centering
        \begin{tikzpicture}[scale=1]
\begin{axis}[%
width=1.2in,
height=1.2in,
scale only axis,
xmin=-0.5,
xmax=4.5,
ymin=-0.5,
ymax=4.5,
ticks = none, 
ticks = none,
axis background/.style={fill=white},
axis line style={draw=none} 
]

\draw[->,>=stealth] (axis cs:0,0) -- (axis cs:4.5,0);
\draw[->,>=stealth] (axis cs:0,0) -- (axis cs:0,4.5);

\addplot[only marks,mark=*,mark size=1.1pt,black
        ]  coordinates {
   (0,0) (1,0) (2,0) (3,0) (4,0) (0,1) (1,1) (2,1) (3,1) (4,1) (0,2) (1,2) (2,2) (3,2) (4,2) (0,3) (1,3) (2,3) (3,3) (4,3) (0,4) (1,4) (2,4) (3,4) (4,4) 
};

\addplot[only marks,mark=*,mark size=3.1pt,myblue
        ]  coordinates {
  (0,0) (1,0) (2,0) (3,0) (4,0)
  (0,1) (1,1) (2,1) (3,1)
  (0,2) (1,2) (2,2)
  (0,3) (1,3)
  (0,4)
};

\end{axis}
\end{tikzpicture} 
        \quad \quad 
        \begin{tikzpicture}[scale=1]
\begin{axis}[%
width=1.2in,
height=1.2in,
scale only axis,
xmin=-0.5,
xmax=4.5,
ymin=-0.5,
ymax=4.5,
ticks = none, 
ticks = none,
axis background/.style={fill=white},
axis line style={draw=none} 
]

\draw[->,>=stealth] (axis cs:0,0) -- (axis cs:4.5,0);
\draw[->,>=stealth] (axis cs:0,0) -- (axis cs:0,4.5);

\addplot[only marks,mark=*,mark size=1.1pt,black
        ]  coordinates {
   (0,0) (1,0) (2,0) (3,0) (4,0) (0,1) (1,1) (2,1) (3,1) (4,1) (0,2) (1,2) (2,2) (3,2) (4,2) (0,3) (1,3) (2,3) (3,3) (4,3) (0,4) (1,4) (2,4) (3,4) (4,4) 
};

\addplot[only marks,mark=*,mark size=3.1pt,myblue
        ]  coordinates {
   (0,0) (1,0) (2,0) (3,0) (4,0) (0,1) (1,1) (2,1) (3,1) (4,1) (0,2) (1,2) (2,2) (3,2) (4,2) (0,3) (1,3) (2,3) (3,3) (4,3) (0,4) (1,4) (2,4) (3,4) (4,4) 
};

\end{axis}
\end{tikzpicture} 
        \quad \quad 
        \begin{tikzpicture}[scale=1]
\begin{axis}[%
width=1.2in,
height=1.2in,
scale only axis,
xmin=-0.5,
xmax=4.5,
ymin=-0.5,
ymax=4.5,
ticks = none, 
ticks = none,
axis background/.style={fill=white},
axis line style={draw=none} 
]

\draw[->,>=stealth] (axis cs:0,0) -- (axis cs:4.5,0);
\draw[->,>=stealth] (axis cs:0,0) -- (axis cs:0,4.5);

\addplot[only marks,mark=*,mark size=1.1pt,black
        ]  coordinates {
   (0,0) (1,0) (2,0) (3,0) (4,0) (0,1) (1,1) (2,1) (3,1) (4,1) (0,2) (1,2) (2,2) (3,2) (4,2) (0,3) (1,3) (2,3) (3,3) (4,3) (0,4) (1,4) (2,4) (3,4) (4,4) 
};

\addplot[only marks,mark=*,mark size=3.1pt,myblue
        ]  coordinates {
  (0,0) (1,0) (2,0) (3,0) (4,0)
  (0,1) (1,1) (2,1) (3,1)
  (0,2) (1,2) (2,2) (3,2)
  (0,3) (1,3) (2,3)
  (0,4)
};

\end{axis}
\end{tikzpicture} 
        \caption{All exponents of \emph{total} (left), \emph{max} (middle) and \emph{Euclidean} (right) degree 4.}
        \label{fig:degrees}
    \end{figure}
    Such matrices $A$ are perhaps the most relevant ones in approximation theory. However, keeping the analogy with toric geometry and applications like \cite{hubert2022sparse} in mind, we will state results for more general (and more \emph{sparse}) choices of $A$ too. 
\end{remark}
Our first result is a sufficient condition for ${\cal X}_{A,\otimes}$ to have the expected dimension.
\begin{theorem} \label{thm:dimtensor}
    If ${\rm rank}(A) = m$, then ${\cal X}_{A,\otimes}$ has dimension $m$. 
\end{theorem}
\begin{proof}
    The (transpose of the) toric Jacobian matrix of ${\cal T}_{A,\otimes}$ is the $m \times n$ matrix 
    \[ J^\top \, = \, \left ( t_i \cdot \frac{\partial {\cal T}_{a_j,\otimes}}{\partial t_i} \right )_{i,j}.\]
    That is, $J$ is the usual Jacobian matrix up to scaling row $i$ by $t_i$. At a point $t$ with non-zero coordinates, the rank of $J$ equals that of the Jacobian. Hence, it suffices to show that for an open subset of such points, $J$ has rank $m$. Indeed, the generic rank of the Jacobian matrix equals the dimension of ${\cal X}_{A,\otimes}$. Let $I \subset \{1, \ldots, n \}$ be an $m$-element subset such that the submatrix $A_I$ consisting of the columns indexed by $I$ is invertible. We claim that 
    \begin{equation}
    \det(J_I^\top) \, = \, 2^{E} \cdot t^{\sum_{j \in I}a_j} \cdot \det(A_I) \, + \, \text{lower degree terms} \, \neq \, 0,
    \end{equation}
    with $E \geq 0$. Here $J_I^\top$ is the submatrix of $J^\top$ indexed by $I$. Hence, ${\rm rank}(J) = m$ for generic $t$. Our claim follows from the fact that $T_k(t) = 2^{\max\{0,k-1\}} \, t^k + \text{lower degree terms}$, so that 
    \[ (J^\top)_{ij} = a_{ij} \cdot 2^{\sum_{k} \max \{a_{kj} - 1, 0 \}} \cdot t^{a_j} + \text{lower degree terms}. \qedhere \]
\end{proof}
\begin{remark}
The converse statement of Theorem \ref{thm:dimtensor} does not hold. For instance, 
\[ A_1 \, = \, \begin{bmatrix}
    1 & 2 & 3 \\ 
    1 & 2 & 3
\end{bmatrix}, \quad \text{and} \quad A_2 \, = \, \begin{bmatrix}
    2 & 2 & 2 \\ 
    2 & 2 & 2
\end{bmatrix}\]
both have rank $1$. However, we have $\dim {\cal X}_{A_1,\otimes} = 2 \neq \dim {\cal X}_{A_2,\otimes} = 1$. 
\end{remark}

Next, we bound the degree of ${\cal X}_{A,\otimes}$. To this end, we consider the following points in $\mathbb{Z}^m$:
\[ B_j \, = \, \{-a_{1j}, a_{1j}\} \times \cdots \times \{-a_{mj}, a_{mj}\}, \quad j = 1, \ldots, n. \]
I.e., $B_j$ contains $\leq 2^m$ points which form the vertices of a box containing the origin in its relative interior. The convex hull of their union is the polytope $P_B = {\rm conv}(B_1 \cup \cdots \cup B_n)$. The intersection of $P_B$ with $\mathbb{R}^m_{\geq 0}$ contains the polytope $P_C$, defined as 
\[P_C = {\rm conv} \left ( \, {\rm Newt}({\cal T}_{a_1,\otimes}) \cup \cdots \cup {\rm Newt}({\cal T}_{a_n,\otimes}) \cup 0 \, \right ).\]
Here ${\rm Newt}(f)$ is the Newton polytope, i.e., the convex hull of the exponents appearing in $f$.
\begin{proposition} \label{prop:degtensor}
    If ${\rm dim} \, {\cal X}_{A,\otimes} = m$, then its degree is bounded by the following quantities: 
    \begin{equation} \label{eq:ineqtensor}
    \deg {\cal X}_{A,\otimes} \, \leq \, m! \cdot {\rm vol}(P_C) \, \leq \,   m! \cdot 2^{-m} \cdot {\rm vol}(P_B). \end{equation}
\end{proposition}
\begin{proof}
    The degree of ${\cal X}_{A,\otimes}$ is the number of solutions to $c_0 + C \cdot x = 0$ and $x \in {\cal X}_{A,\otimes}$, for generic $c_0, C$. This is at most the number of points $t \in \mathbb{C}^m$ such that $c_0 + C \cdot {\cal T}_{A,\otimes}(t) = 0$, with equality if $\deg {\cal T}_{A,\otimes} = 1$. Kushnirenko's theorem \cite[Theorem 3.16]{telen2022introduction} implies the inequality $\deg {\cal X}_{A,\otimes}  \leq  m! \cdot {\rm vol}(P_C)$. The second inequality, i.e., $m! \cdot {\rm vol}(P_C)  \leq  m! \cdot 2^{-m} \cdot {\rm vol}(P_B)$ follows from the symmetry of $P_B$, and the easy fact that $P_C \subset P_B \cap \mathbb{R}^m_{\geq 0}$. For what follows, it is instructive to also prove the inequality $\deg {\cal X}_{A,\otimes}  \leq m! \cdot 2^{-m} \cdot {\rm vol}(P_B)$ more directly. 
    
    Next to \eqref{eq:chebpol1}, the Chebyshev polynomials $T_k(t)$ can be characterized by the property
    \[ T_{k} \left ( \frac{v + v^{-1}}{2} \right ) \, = \, \frac{v^{k} +v^{-k}}{2}, \quad \text{for} \, \, \, v \in \mathbb{C} \setminus \{0\}. \]
    Therefore, ${\cal X}_{A,\otimes}$ is alternatively parametrized as follows:
    \[ \psi: (v_1, \ldots, v_m) \, \longmapsto \, \left( \, \prod_{i = 1}^m \left ( \frac{v_i^{a_{i1}} + v_i^{-a_{i1}}}{2} \right ), \, \ldots, \,  \prod_{i = 1}^m \left ( \frac{v_i^{a_{in}} + v_i^{-a_{in}}}{2} \right ) \, \right). \]
    The degree of ${\cal X}_{A,\otimes}$ is 
    the number of points $v \in (\mathbb{C} \setminus \{0\})^m$ such that $c_0 + C \cdot \psi(v) = 0$, divided by the degree of $\psi$. By Kushnirenko's theorem \cite[Theorem 3.16]{telen2022introduction}, the former is at most $m! \cdot {\rm vol}(P_B)$. Since $\psi(v_1, \ldots, v_m) = \psi(v_1^{\pm 1}, \ldots, v_m^{\pm 1})$, the degree of $\psi$ is at least $2^m$.
\end{proof}

\begin{example} \label{ex:tensor1}
Consider again the matrix $A$ from Example \ref{ex:pringles1}. The Chebyshev surface ${\cal X}_{A,\otimes}$ has degree 7. It is defined by the following degree 7 equation in 3 variables:
\[ -6 x^4 y+x^3 z-x^2 y \left(-48 y^4+22 y^2-3\right)-x y^2 \left(20 y^2-3\right) z+y^3
   \left(-16 y^4+8 y^2+2 z^2-1\right)\,=\,0\]
Its real picture is shown in Figure \ref{fig:sub-second}. The polytopes $P_B$ and $P_C$ are shown in Figure \ref{fig:tensor1}.
\begin{figure}
    \centering
     \begin{tikzpicture}[scale=0.8]
\begin{axis}[%
width=1.2in,
height=1.8in,
scale only axis,
xmin=-2.5,
xmax=2.5,
ymin=-3.5,
ymax=3.5,
ticks = none, 
ticks = none,
axis background/.style={fill=white},
axis line style={draw=none} 
]

\draw[->,>=stealth] (axis cs:0,0) -- (axis cs:2.5,0);
\draw[->,>=stealth] (axis cs:0,0) -- (axis cs:0,3.5);

\addplot[only marks,mark=*,mark size=1.1pt,black
        ]  coordinates {
   (0,0) (1,0) (2,0) (0,1) (1,1) (2,1) (0,2) (1,2) (2,2) (0,3) (1,3) (2,3)
   (0,0) (-1,0) (-2,0) (0,1) (-1,1) (-2,1) (0,2) (-1,2) (-2,2) (0,3) (-1,3) (-2,3)
   (0,0) (1,0) (2,0) (0,-1) (1,-1) (2,-1) (0,-2) (1,-2) (2,-2) (0,-3) (1,-3) (2,-3)
   (-1,-1) (-2,-1) (-1,-2) (-2,-2) (-1,-3) (-2,-3)
};

\addplot [color=myorange,solid,fill opacity=0.2,fill = myorange,forget plot]
  table[row sep=crcr]{%
 2 -3\\
2 3\\	
-2 3\\
-2 -3 \\
2 -3\\
};

\addplot [color=myblue,solid,fill opacity=0.4,fill = myblue,forget plot]
  table[row sep=crcr]{%
 0 0\\
0 3\\	
2 3\\
2 1 \\
1 0 \\
0 0\\
};

\addplot [very thick, color=myblue,solid,fill opacity=0.2,fill = myblue,forget plot]
  table[row sep=crcr]{%
 0 0 \\
 0 3\\
};

\addplot [very thick, color=myblue,solid,fill opacity=0.2,fill = myblue,forget plot]
  table[row sep=crcr]{%
 0 3 \\
 2 3 \\
};

\addplot [very thick, color=myblue,solid,fill opacity=0.2,fill = myblue,forget plot]
  table[row sep=crcr]{%
 2 3\\
 2 1 \\
};

\addplot [very thick, color=myblue,solid,fill opacity=0.2,fill = myblue,forget plot]
  table[row sep=crcr]{%
 2 1\\
 1 0 \\
};

\addplot [very thick, color=myblue,solid,fill opacity=0.2,fill = myblue,forget plot]
  table[row sep=crcr]{%
 1 0\\
 0 0 \\
};

\addplot [very thick, color=myorange,solid,fill opacity=0.2,fill = myorange,forget plot]
  table[row sep=crcr]{%
 2 -3 \\
 2 3 \\
};

\addplot [very thick, color=myorange,solid,fill opacity=0.2,fill = myorange,forget plot]
  table[row sep=crcr]{%
 2 3\\
 -2 3 \\
};

\addplot [very thick, color=myorange,solid,fill opacity=0.2,fill = myorange,forget plot]
  table[row sep=crcr]{%
 -2 3\\
 -2 -3 \\
};

\addplot [very thick, color=myorange,solid,fill opacity=0.2,fill = myorange,forget plot]
  table[row sep=crcr]{%
 -2 -3\\
 2 -3 \\
};

\end{axis}
\end{tikzpicture} 
     \quad \quad \quad
    \begin{tikzpicture}[scale=0.8]
\begin{axis}[%
width=1.2in,
height=1.8in,
scale only axis,
xmin=-2.5,
xmax=2.5,
ymin=-3.5,
ymax=3.5,
ticks = none, 
ticks = none,
axis background/.style={fill=white},
axis line style={draw=none} 
]

\draw[->,>=stealth] (axis cs:0,0) -- (axis cs:2.5,0);
\draw[->,>=stealth] (axis cs:0,0) -- (axis cs:0,3.5);

\addplot[only marks,mark=*,mark size=1.1pt,black
        ]  coordinates {
   (0,0) (1,0) (2,0) (0,1) (1,1) (2,1) (0,2) (1,2) (2,2) (0,3) (1,3) (2,3)
   (0,0) (-1,0) (-2,0) (0,1) (-1,1) (-2,1) (0,2) (-1,2) (-2,2) (0,3) (-1,3) (-2,3)
   (0,0) (1,0) (2,0) (0,-1) (1,-1) (2,-1) (0,-2) (1,-2) (2,-2) (0,-3) (1,-3) (2,-3)
   (-1,-1) (-2,-1) (-1,-2) (-2,-2) (-1,-3) (-2,-3)
};

\addplot[only marks,mark=*,mark size=3.1pt,mygreen
        ]  coordinates {
  (2,3) (-2,3) (2,-3) (-2,-3)
}; \label{greenpts}

\addplot[only marks,mark=*,mark size=3.1pt,mypurple
        ]  coordinates {
  (1,1) (-1,1) (-1,-1) (1,-1)
}; \label{purplepts}

\addplot[only marks,mark=*,mark size=3.1pt,myyellow
        ]  coordinates {
  (1,2) (-1,2) (-1,-2) (1,-2)
}; \label{yellowpts}

\end{axis}
\end{tikzpicture} 
    \quad \quad \quad
    \begin{tikzpicture}[scale=0.8]
\begin{axis}[%
width=1.2in,
height=1.8in,
scale only axis,
xmin=-2.5,
xmax=2.5,
ymin=-3.5,
ymax=3.5,
ticks = none, 
ticks = none,
axis background/.style={fill=white},
axis line style={draw=none} 
]

\draw[->,>=stealth] (axis cs:0,0) -- (axis cs:2.5,0);
\draw[->,>=stealth] (axis cs:0,0) -- (axis cs:0,3.5);

\addplot[only marks,mark=*,mark size=1.1pt,black
        ]  coordinates {
   (0,0) (1,0) (2,0) (0,1) (1,1) (2,1) (0,2) (1,2) (2,2) (0,3) (1,3) (2,3)
   (0,0) (-1,0) (-2,0) (0,1) (-1,1) (-2,1) (0,2) (-1,2) (-2,2) (0,3) (-1,3) (-2,3)
   (0,0) (1,0) (2,0) (0,-1) (1,-1) (2,-1) (0,-2) (1,-2) (2,-2) (0,-3) (1,-3) (2,-3)
   (-1,-1) (-2,-1) (-1,-2) (-2,-2) (-1,-3) (-2,-3)
};

\addplot[only marks,mark=*,mark size=3.1pt,myblue
        ]  coordinates {
  (0,0) (1,1) 
};

\addplot [very thick, color=myblue,solid,fill opacity=0.2,fill = myblue,forget plot]
  table[row sep=crcr]{%
 1 0\\
 1 2 \\
};

\addplot [color=myblue,solid,fill opacity=0.2,fill = myblue,forget plot]
  table[row sep=crcr]{%
 0 1\\
2 1\\	
2 3\\
0 3 \\
0 1\\
};

\end{axis}
\end{tikzpicture} 
    \caption{Left: the polygons $P_C \subset P_B$. Here $P_B$ is the convex hull of $B_1$ (\ref{yellowpts}), $B_2$ (\ref{purplepts}) and $B_3$ (\ref{greenpts}). Right: $P_C$ is the convex hull of a line segment, two points, and a square. }
    \label{fig:tensor1}
\end{figure}
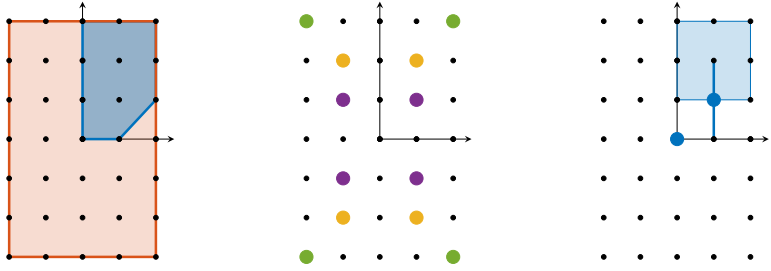
The chain of inequalities \eqref{eq:ineqtensor} for this example reads $7 \leq 11 \leq 12$.
\end{example}
While Example \ref{ex:tensor1} shows that the bounds of Proposition \ref{prop:degtensor} may be strict, the inequalities in the statement are in fact equalities if $A$ is \emph{sufficiently dense}, meaning that
\begin{equation} \label{eq:suffdense} 
\text{ for all } a_j \in A \text{ and all } i = 1, \ldots, m, \text{ we have } a_j - e_i \in A \text{ if } a_{ij} > 0. \end{equation}
Here $e_i$ is the $i^{th}$ standard basis vector of $\mathbb{R}^m$, and ``$a_j \in A$'' means ``$a_j$ is among the columns of $A$''. Notice that this property holds for any of the choices for $A$ in Remark \ref{rem:densechoices}. 
\begin{theorem} \label{thm:suffdense}
    If $A$ has rank $m$ and it satisfies \eqref{eq:suffdense}, then $P_A = P_C = P_B \cap \mathbb{R}^m_{\geq 0}$ and 
    \[ \deg {\cal X}_{A,\otimes} \, = \, m! \cdot {\rm vol}(P_A) \, = \,   m! \cdot 2^{-m} \cdot {\rm vol}(P_B). \]
\end{theorem}
\begin{proof}
    The equality of polytopes is straightforward and left to the reader. Because of \eqref{eq:suffdense}, the vector space spanned by $\{ {\cal T}_{a_j, \otimes} \}_{j = 1, \ldots, n}$ is equal to the span of $\{ t^{a_j} \}_{j = 1, \ldots, n}$. That means ${\cal X}_{A,\otimes}$ is isomorphic to the toric variety ${\cal Y}_A$ via a linear change of coordinates. Therefore $\deg {\cal X}_{A,\otimes} = \deg {\cal Y}_A$, and \eqref{eq:degtoric} gives the desired degree formula (\eqref{eq:suffdense} implies $\deg \Phi_A = 1$). 
\end{proof}
The sufficient condition \eqref{eq:suffdense} is not necessary for the conclusion of Theorem \ref{thm:suffdense}.
\begin{example}\label{ex:m3n6}
    We consider the $3$-dimensional Chebyshev variety ${\cal X}_{A,\otimes} \subset \mathbb{C}^6$ obtained from 
    \begin{equation} \label{eq:Adim3}
    A \, = \, \begin{bmatrix}
        1 & 0 & 0 & 2 & 2 & 0\\
        0 & 1 & 0 & 2 & 0 & 2\\
        0 & 0 & 1 & 0 & 2 & 2
    \end{bmatrix}.
    \end{equation}
    Using any computer algebra system, like for instance \texttt{Macaulay2} \cite{M2} or \texttt{Oscar.jl} \cite{OSCAR}, one can verify that ${\cal X}_{A,\otimes}$ has degree 28 and its ideal  is generated by
    \[\begin{matrix}
        2\,x_{2}^{2}x_{5}-2\,x_{1}^{2}x_{6}-x_{5}+x_{6}, \quad \,2\,x_{3}^{2}x_{4}-2\,x_{1}^{2}x_{6}-x_{4}+x_{6}, \quad \,4\,x_{2}^{2}x_{3}^{2}-2\,x_{2}^{2}-2\,x_{3}^{2}-x_{6}+1,\\
        \,4\,x_{1}^{2}x_{3}^{2}-2\,x_{1}^{2}-2\, x_{3}^{2}-x_{5}+1, \quad \,4\,x_{1}^{2}x_{2}^{2}-2\,x_{1}^{2}-2\,x_{2}^{2}-x_{4}+1.
    \end{matrix}
    \]
    We have $P_A = P_C = P_B \cap \mathbb{R}^m_{\geq 0}$, and all inequalities in Proposition \ref{prop:degtensor} hold as equalities.
\end{example}
    We now improve the bounds from Proposition \ref{prop:degtensor} for the case $m=2$. We write
    \begin{equation} \label{eq:Asurface} A \, = \, \begin{bmatrix}
        a_1 & a_2 & a_3 & \cdots & a_n\\ b_1 & b_2 & b_3 & \cdots & b_n
    \end{bmatrix} \, \, \in \, \mathbb{N}^{2\times n}.\end{equation}
    Without loss of generality, we assume that $a_1 \leq a_2 \leq \cdots \leq a_n$. 
    Let $\sigma \in S_n$ be a permutation for which $b_{\sigma(1)} \leq b_{\sigma(2)} \leq \cdots \leq b_{\sigma(n)}$.
    We define $\alpha$ to be the largest $0 \leq j \leq n$ such that $b_j \neq b_n$, where $b_0$ is defined to be $0$. Similarly, $\beta$ is the largest $j$ such that $a_{\sigma(j)} \neq a_{\sigma(n)}$, and $a_{\sigma(0)} =0$. Our running example $A = \left [ \begin{smallmatrix}
        1 & 1 & 2 \\ 2 & 1 & 3
    \end{smallmatrix} \right ]$ has 
    $ \alpha = 2, \,\sigma = (2,1,3), \, \beta = 2$.
\begin{theorem} \label{thm:betterbound}
    If $A$ from \eqref{eq:Asurface} has rank $2$, then the degree of the tensor product Chebyshev surface ${\cal X}_{A,\otimes} \subset \mathbb{C}^n$ is at most $\frac{1}{4}(2 \, {\rm vol}(P_B) - 4 b_n (a_n-a_\alpha) - 4 a_{\sigma(n)}(b_{\sigma(n)}-b_{\sigma(\beta)}))$.
\end{theorem}
\begin{proof}
    This proof requires some projective toric geometry. We have seen in the proof of Proposition \ref{prop:degtensor} that $\deg {\cal X}_{A,\otimes}$ is at most 1/4 times the number of points $v \in (\mathbb{C} \setminus \{0\})^2$ such that $c_0 + C \cdot \psi(v) = 0$. Kushnirenko's number $2 \, {\rm vol}(P_B)$ counts the number of points defined by the corresponding ``homogeneous'' equations on the toric compactification $X_{P_B}$ of $(\mathbb{C} \setminus \{0\})^2$ associated with $P_B$, see \cite[Section 3.8, Example 6.30]{telen2022introduction}. We may subtract from this quantity all solutions that lie in the boundary $X_{P_B} \setminus (\mathbb{C} \setminus \{0\})^2$. Let $p,q$ be coordinates on $\mathbb{R}^2 \supset P_B$. From a standard local analysis, one sees that the face $P_B\cap \{p=a_n\}$ corresponds to a torus orbit with $2b_n$ spurious solutions, each of them having multiplicity at least $(a_n-a_\alpha)$. The opposite face $P_B\cap \{p=-a_n\}$ contributes at least another $2b_n(a_n-a_\alpha)$ points, which gives the term $- 4 b_n (a_n-a_\alpha)$ in our bound. Similarly, the faces $P_B\cap \{q=\pm b_{\sigma(n)}\}$ lead to at least $4 a_{\sigma(n)}(b_{\sigma(n)}-b_{\sigma(\beta)})$ points in $X_{P_B} \setminus (\mathbb{C} \setminus \{0\})^2$, counting multiplicities.
\end{proof}

\begin{example}
    The bound in Theorem \ref{thm:betterbound} predicts the correct number 7 for the degree of ${\cal X}_{A,\otimes}$ from our running example, see Example \ref{ex:tensor1}. Such efficient bounds are useful for counting solutions to bivariate Chebyshev equations $f_1=f_2=0$ and for finding implicit equations like that in Example \ref{ex:tensor1} via interpolation. 
\end{example}

The bound from Theorem \ref{thm:betterbound} is significantly better than that of Proposition \ref{prop:degtensor} for many Chebyshev surfaces. However, it is still a strict upper bound in general. The reason can be that the degree of $\psi$ is greater than 4 (try $A = \left [ \begin{smallmatrix}
    2 & 0 & 2\\0&2&2
\end{smallmatrix}\right ]$) or that the multiplicities of the ``spurious solutions'' in our proof are underestimated (try $A = \left [ \begin{smallmatrix}
    1 & 2 & 4\\0&2&4
\end{smallmatrix}\right ]$). The combinatorial strategy used in our proof can be adapted to predict the correct degree on a case-by-case~basis.

\section{Cosines of linear forms} \label{sec:5}

The Chebyshev polynomials satisfy $T_k ( \cos(u)) = \cos(k u)$. A $T$-curve from Section \ref{sec:3} can therefore be interpreted as the image of the line through the origin parametrized by $(a_1\, u, a_2 \, u, \ldots, a_n \, u)$ under the coordinate-wise cosine map. This section generalizes this to higher $m$. For $m\geq 1$ and $a_j = (a_{1j}, \ldots,a_{mj}) \in \mathbb{N}^m$, we define the function ${\cal T}_{a_j,\cos{}}$ as follows:
\[
{\cal T}_{a_j,\cos{}}(u) \, = \, \cos(a_j \cdot u) \,  = \,  \cos(a_{1j}u_1 + \ldots + a_{mj}u_m).
\]
Using this definition, we construct and study a second type of parametric varieties. Given a matrix $A = [a_1~\cdots~a_n] \in \mathbb{N}^{m\times n}$, define the map ${\cal T}_{A,\cos{}} : \mathbb{C}^m \to \mathbb{C}^n$, with
\[
{\cal T}_{A,\cos{}}(u) \, = \, (\, \cos(a_1 \cdot u),\, \ldots,\, \cos(a_n \cdot u) \,).
\]
The Zariski closure of the image is the \emph{(cosine) Chebyshev variety} $\Xcos$.
Linear sections of these Chebyshev varieties correspond to equations $f_1 = \ldots = f_m = 0$, where the $f_i$ are of the form $c_0 + \sum_{j=1}^n c_{i,j} \cos(a_{j}\cdot u)$.
In the case $m=1$, we have ${\cal X}_{A,\cos}={\cal X}_{A,T}$ (see Section \ref{sec:3}).

\begin{remark}
    The map ${\cal T}_{A,\cos{}}$ is not polynomial. However, we will see below that $\Xcos$ is the image of a rational parametrization \eqref{eq:reparamCos}, up to Zariski closure. 
\end{remark}

Like the Chebyshev $T$-curve, the Chebyshev variety $\Xcos$ is the image of a linear subspace under the coordinate-wise cosine map.
Indeed, the row-span of the matrix $A$ consists of points $A^\top u = (a_1\cdot u, \ldots, a_n \cdot u)$ for $u\in \mathbb{C}^m$. The map ${\cal T}_{A,\cos{}}$ applies the cosine to each coordinate. Analogously, the toric variety ${\cal Y}_A$ can be viewed as the image of the same linear subspace under the coordinate-wise exponential map. From this perspective, the map $A$ need not be filled with nonnegative integer numbers. It is sufficient that $A$ has rational entries. Theorems \ref{thm:cosineChebDim} and \ref{thm:degCos} assume that $A \in \mathbb{N}^{m \times n}$ is an integer representative of its row-span.

\begin{theorem}\label{thm:cosineChebDim}
    The Chebyshev variety $\Xcos$ is irreducible of dimension equal to $\operatorname{rank}(A)$. In particular, it is the Zariski closure of the projection of the following variety to $\mathbb{C}^n$:
    \[
    {\cal V} \, = \, \{ (x,y) \in \mathbb{C}^n \times (\mathbb{C}^*)^n \,|\, y \in {\cal Y}_A, \, y_j^2-2 x_j y_j + 1 = 0 \text{ for } j = 1,\ldots,n \}.
    \]
\end{theorem}
\begin{proof}
    Denote by $\pi_1: \mathbb{C}^n \times (\mathbb{C}^*)^n \to \mathbb{C}^n$ the projection onto the first $n$ coordinates. We begin by proving that $\Xcos \subset \overline{\pi_1({\cal V})}$. 
    We reparametrize $\Xcos$ as follows:
    \begin{equation}\label{eq:reparamCos}
    \psi \, : \, v \, \longmapsto \,  \left( \frac{v^{a_1}+ v^{-a_1}}{2}, \ldots , \frac{v^{a_n}+ v^{-a_n}}{2} \right), \quad \text{for }\, \, \, v \in (\mathbb{C}\setminus \{0\})^m.
    \end{equation}
    where $v^{a_j} = v_1^{a_{1,j}}\cdots v_m^{a_{m,j}}$, and $v_j = e^{\sqrt{-1} u_j}$. From the change of variables $y_j = v^{a_j}$, we see that a point $x \in \operatorname{im} \psi$ satisfies $x_j = \frac{1}{2}(v^{a_j}+ v^{-a_j}) = \frac{1}{2}(y_j+ y_j^{-1})$, hence it belongs to $\pi_1({\cal V})$. This gives the inclusion $\operatorname{im} \psi \subset \pi_1({\cal V})$ and taking the Zariski closure proves our first claim.
    
    Next, we show that $\dim \Xcos = \dim {\cal Y}_A$. In analogy with the tensor product case, let $J$ be the Jacobian of $\psi$ up to scaling row $i$ by $v_i$. We have
    \[
    J^\top \, = \, \left ( v_i \cdot \frac{\partial \psi_{j}}{\partial v_i} \right )_{i,j} \,=\, \left( a_{i,j}\, \frac{v^{a_j}- v^{-a_j}}{2} \right)_{i,j}.
    \]
    When $v^{2a_j}\neq 1$ for $j = 1,\ldots, n$,  this has the same rank as $A$. Hence $\dim \Xcos = \dim {\cal Y}_A$.

    Finally, the other projection $\pi_2:\mathbb{C}^n \times (\mathbb{C}^*)^n \to (\mathbb{C}^*)^n$ maps ${\cal V}$ onto ${\cal Y}_A\cap (\mathbb{C} \setminus \{0\})^n$, and its fibers consist of exactly one point. Since ${\cal Y}_A$ is irreducible of dimension $\operatorname{rank}(A)$, so is ${\cal V}$. As a consequence, $\pi_1({\cal V})$ is irreducible and $\dim \pi_1({\cal V}) \leq \dim {\cal Y}_A = \dim \Xcos$. By the inclusion $\operatorname{im} \psi \subset \pi_1({\cal V})$ proved above, $\Xcos$ coincides with the Zariski closure of $\pi_1({\cal V})$.
\end{proof}

A useful consequence of Theorem \ref{thm:cosineChebDim} is that we can obtain the equations of $\Xcos$ from the equations of ${\cal Y}_A$ just by substituting $y_j \mapsto x_j \pm \sqrt{x_j^2-1}$ in every coordinate, see Example \ref{ex:elliptope} below. We now deduce the degree of ${\cal X}_{A,\cos}$. Recall from Section \ref{sec:2} that $\deg \Phi_A = [\mathbb{Z}^m : \mathbb{Z}A]$ denotes the index of the sublattice of $\mathbb{Z}^m$ generated by the columns of the matrix $A$. We write $A \cup -A$ for the set of points $\{ a_j, - a_j \, : \, a_j \in A \} \subset \mathbb{Z}^n$, and $P_{A,\cos{}}$ is their convex hull.

\begin{theorem}\label{thm:degCos}
    Let $P_{A,\cos{}} = \operatorname{conv} (A\cup -A) \subset \mathbb{R}^m$ and let $\pi_1 : {\cal V} \to \Xcos$ be the projection from Theorem \ref{thm:cosineChebDim}. If $\dim \Xcos = m$, then the degree of $\Xcos$ is 
    \[
    \frac{m!\, \operatorname{vol} P_{A,\cos{}}}{\deg \pi_1 \cdot \deg \Phi_A}.
    \]
\end{theorem}
\begin{proof}
    Let us compute the degree of $\Xcos$ as the number of its intersection points with a generic affine subspace $L\subset \mathbb{C}^n$ of codimension $m$. Notice that the pullback of $L$ via $\pi_1$ intersects ${\cal V}$ transversally and the number of intersection points is the degree of ${\cal V}$. As a consequence, we obtain that  $\deg \Xcos = \frac{\deg {\cal V}}{\deg \pi_1}$. 

    It remains to compute $\deg {\cal V}$. Using \eqref{eq:reparamCos}, we can express ${\cal V}$ as the image of the map
    \[
    (\mathbb{C}^*)^m \ni v \mapsto \left( \frac{v^{a_1}+ v^{-a_1}}{2}, \ldots , \frac{v^{a_n}+ v^{-a_n}}{2}, v^{a_1}, \ldots , v^{a_n} \right) \in \mathbb{C}^{2n}.
    \]
    Its degree is the number of solutions to the following system of Laurent polynomial equations:
    \begin{equation}\label{eq:sysDegCosProof}
            \sum_{j=1}^n \left (  c_{i,j} \frac{v^{a_j}+ v^{-a_j}}{2} \, + \,  d_{i,j} v^{a_j} \right ) \, + \,  c_{i,0} =0 \qquad i = 1,\ldots,m,
    \end{equation}
    where $c_{i,0}, c_{i,j}, d_{i,j}$ are generic. Reorganizing the monomials, the number of solutions to \eqref{eq:sysDegCosProof} is the same as the number of solutions to the system $\sum_{j=1}^n c'_{i,j} v^{-a_j} + d'_{i,j} v^{a_j} + c_0=0$ for $ i = 1,\ldots,m$. By Kushnirenko's theorem \cite[Theorem 3.16]{telen2022introduction}, this number is the normalized volume of $P_{A,\cos{}} = \operatorname{conv} (A\cup -A)$, divided by the lattice index of $A$ (or, more precisely, of $A\cup-A$, but these coincide). The proposition follows from these observations. 
\end{proof}

We note that the degree $\deg \pi_1$ of the projection $\pi_1$ is at least $2$, since the preimage of a point $x = \pi_1(x,y) \in \Xcos$ contains $(x,y)$ and $(x,y^{-1})$. Hence, $\frac{{\rm vol}(P_{A,\cos})}{2 \deg \Phi_A}$ is an upper bound for $\deg {\cal X}_{A,\cos}$. However, we might have $\deg \pi_1 >2$.

\begin{example}
    Let $A = \left[\begin{smallmatrix}  1 & 0 & 0 & 2 \\ 0 & 1 & 0 & 3 \\ 0 & 0 & 1 & 0 \end{smallmatrix}\right]$. The hypersurface $\Xcos$ is the zero locus of 
    $$
    4x_1^4 - 16x_1^2 x_2^3x_4 + 12 x_1^2 x_2x_4 - 4 x_1^2 + 16x_2^6 - 24 x_2^4 + 8x_2^3x_4 + 9x_2^2 - 6x_2x_4 + x_4^2,
    $$
    so it has degree $6$. The volume of $P_{A,\cos}$ is $4$ and $\deg \Phi_A = 1$. The degree of the map $\pi_1$ is $4$, since $\pi_1^{-1}(\pi_1(x,y)) = \{(x,y),(x,y^{-1}),(x, y_1,y_2,y_3^{-1},y_4),(x, y_1^{-1},y_2^{-1},y_3,y_4^{-1})\}$ for a generic point $(x,y) \in {\cal V}$. Hence the formula from Theorem \ref{thm:degCos} reads $\frac{3! \cdot 4}{4 \cdot 1} = 6$.
\end{example}

\begin{example}\label{ex:elliptope}
    Consider the matrix $A$ from Example \ref{ex:pringles1}. The corresponding toric variety is the surface in Figure \ref{fig:degpringles}. The variety ${\cal V}$ from Theorem \ref{thm:cosineChebDim} is the zero locus of the polynomials
    \[
    y_1 y_2 - y_3,\; y_1^2-2 x_1 y_1 + 1,\; y_2^2-2 x_2 y_2 + 1,\; y_3^2-2 x_3 y_3 + 1.
    \]
    After elimination of the $y$ coordinates, we obtain the defining equation for $\Xcos:$
    \[
    \Xcos = \{ x \in \mathbb{C}^3 \,|\, x_1^2 + x_2^2 + x_3^2 - 2 x_1 x_2 x_3 - 1 = 0\}.
    \]
    This surface, displayed in Figure \ref{fig:sub-third}, describes the locus of $3\times 3$ correlation matrices with zero determinant. It bounds a spectrahedron known as the \emph{elliptope}, see \cite[Figure 1.1]{michalek2021invitation} and the related examples.
    As previously noticed, we can also obtain the equation for $\Xcos$ by substitution. There are $2^3$ possible substitutions, depending on the signs of the square roots. If we multiple them all, we obtain the desired equation, up to removing exponents:
    \begin{gather*}
    \prod_{\sigma_i \in \{\pm 1\}} \left(x_1 + \sigma_1\sqrt{x_1^2-1}\right)\left(x_2 + \sigma_2\sqrt{x_2^2-1}\right)-\left(x_3 + \sigma_3\sqrt{x_3^2-1}\right) \\
    = 16 \left(x_1^2+x_2^2+x_3^2-2 x_1 x_2 x_3 -1\right)^2.
    \end{gather*}
    In this case, the polytope $P_{A,\cos{}}$ is a hexagon. Its vertices are $\pm(1,2), \pm(1,1), \pm(2,3)$. The Euclidean volume is $3$. The lattice generated by the columns of $A$ has index $1$, and the degree of $\pi_1$ is $2$. Hence, the degree formula from Theorem \ref{thm:degCos} confirms that $\deg \Xcos = 3$.
    This is the number of intersection points of $\Xcos$ with a generic line $C \cdot x + c_0 = 0$. The three intersection points are in bijection with the three pairs of solutions $(v, v^{-1})$ to the equations
    \begin{align*}
        c_{i,0} \, + \, c_{i,1} \frac{v_1 v_2^2 + v_1^{-1} v_2^{-2}}{2} \, + \, c_{i,2} \frac{v_1 v_2 + v_1^{-1} v_2^{-1}}{2} \, + \, c_{i,3} \frac{v_1^2 v_2^3 + v_1^{-2} v_2^{-3}}{2} &\, = \,  0, \quad i = 1, 2.
    \end{align*}
    Solving such equations using homotopy continuation is the topic of Section \ref{sec:6-3}. Undoing the coordinate change $v_j = e^{\sqrt{-1}u_j}$ by setting $u_j = - \sqrt{-1} \log (v_j)$ gives all solutions to \[c_{i,0} + c_{i,1} \cos(u_1 + 2u_2) + c_{i,2} \cos(u_1 + u_2) + c_{i,3} \cos(2u_1 + 3u_2) \, = \,  0, \quad i = 1, 2. \qedhere \]
\end{example}

Next, we describe the singular locus of cosine Chebyshev varieties. We will do this via the parametrization ${\cal T}_{A,\cos}$. This has the disadvantage that we can only describe singularities inside the dense subset ${\rm im }\, {\cal T}_{A,\cos}\subset \Xcos$. For now, we do not know whether taking the closure gives rise to more singular points. In the examples we checked, this was not the case.

As noted above, we can work in a slightly more general setting where the matrix $A$ has rational entries. In this way, we can transform any rank $m$ matrix into its canonical form, where the leftmost $m\times m$ block is the identity matrix $\mathds{1}_m$, and the remaining columns have entries in $\mathbb{Q}$. Without loss of generality, our next result assumes that $A$ is in this form. This helps to parametrize the singular locus of $\operatorname{im} {\cal T}_{A,\cos}$. We use the notation $[n]=\{1, \ldots ,n \}$. For $J \subset [n]$, $A_J$ is the submatrix of $A$ consisting of the columns indexed by $J$. For $I\subset [m]$, $(a_j)_I$ denotes the vector of length $|I|$ containing the entries of $a_j$ indexed by $I$.

\begin{proposition}\label{prop:singCos}
    Let $A = \left[\mathds{1}_m~|~a_{m+1}~\cdots~a_{n} \right] \in \mathbb{Q}^{m\times n}$. We define two families of affine subspaces in $\mathbb{C}^m$. First, let ${\cal J} = \{J\subset [n]\,:\, {\rm rank}( A_J) < m \}$ and define
    \begin{equation}\label{eq:singCosPoints}
    L_{J,\kappa} \, =\, \{ u \in \mathbb{C}^m \,:\,  a_j \cdot u = \kappa_j \pi \, \,  \hbox{ for } \,  j \not\in J \}, \quad J \in {\cal J}, \, \kappa \in \mathbb{Z}^{n-|J|}.
    \end{equation}
    The second family of affine spaces is given by
    \begin{equation}\label{eq:singCosSelfInt}
    H_{I,\nu} \, = \, \left\lbrace u \in \mathbb{C}^m \,:\, (a_j)_I \cdot u_I = \nu_j\pi \hbox{ for } j = m+1,\ldots, n \right\rbrace, \quad I \subsetneq [m], \, \nu \in \mathbb{Z}^{n-m}. 
    \end{equation}
    The singular locus of $\operatorname{im} {\cal T}_{A,\cos}$ is contained in the union of all ${\cal T}_{A,\cos}(L_{J,\kappa})$ and ${\cal T}_{A,\cos}(H_{I,\nu})$.
\end{proposition}

Although the list of affine spaces $L_{J,\kappa}, H_{I,\nu}$ in Proposition \ref{prop:singCos} is infinite, there are only finitely many distinct images under the map ${\cal T}_{A,\cos}$. That is, it is sufficient to consider finitely many values of $\kappa$ and $\nu$. Also, among the affine spaces $L_{J,\kappa}$, it suffices to use only the maximal elements of ${\cal J}$ (with respect to inclusion). In the language of matroids, this means we use only the maximal dependent sets $J$ of the matroid of $A$.

\begin{proof}[Proof of Proposition \ref{prop:singCos}]
    The singular locus of ${\cal T}_{A,\cos{}}(\mathbb{C}^m)$ is contained in the union of the image of points where the Jacobian of ${\cal T}_{A,\cos{}}$ is rank-deficient, and its self-intersections. 

    The Jacobian matrix of ${\cal T}_{A,\cos{}}$ is $J^\top = \left[ a_1 \sin{(a_1 \cdot u)} \cdots a_n \sin{(a_n \cdot u)}\right]$. It is rank deficient if and only if $\sin{(a_j \cdot u)}=0$ for all $j \not\in J$, where $J \in {\cal J}$. This recovers $L_{J,\kappa}$.

    Now, let $u, u' \in \mathbb{C}^m$ be such that  $u\neq \pm u' \mod 2\pi$, and ${\cal T}_{A,\cos{}}(u) = {\cal T}_{A,\cos{}}(u')$.
    Since the left-most $m\times m$ submatrix of $A$ is the identity, these points must satisfy $\cos(u_i) = \cos(u'_i)$ for all $i=1,\ldots,m$. Therefore, there exists $I\subset [m]$ such that 
    \[
    u_i = -u'_i \hbox{ for } i \not\in I, \qquad u_i = u'_i \hbox{ for } i \in I.
    \]
    modulo $2\pi$. Substituting this in the remaining entries of ${\cal T}_{A,\cos{}}(u')$ we obtain
    \[
    \cos{(a_j \cdot u')} = \cos{((a_j)_I \cdot u_I -(a_j)_{[m]\setminus I} \cdot u_{[m]\setminus I})}, \quad j = m+1,\ldots ,n.
    \]
    Using $\cos{(a_j \cdot u')} = \cos{(a_j \cdot u)}$, we find that $(a_j)_I \cdot u_I = \nu_j\pi$ or $(a_j)_{[m]\setminus I} \cdot u_{[m]\setminus I} = \nu_j\pi$ for some $\nu_j$, and for all $j = m+1,\ldots,n$. These are the spaces $H_{I,\nu}$ and $H_{[m]\setminus I,\nu}$.
\end{proof}

\begin{remark}
    It follows from Proposition \ref{prop:singCos} that if $A$ is generic, then
    the singular locus has codimension at least $\min \{n-m, m+1\}$ in $\operatorname{im} {\cal T}_{A,\cos}$.
\end{remark}

The description of the singular locus in Proposition \ref{prop:singCos}
sends \emph{affine-linear} spaces through the coordinatewise cosine map, rather than just linear spaces. This suggests an extension of the definition of cosine Chebyshev varieties, to be possibly investigated in future work.

In the case of surfaces in $3$-space $(m=2,n=3)$, we use the parametrization from Proposition \ref{prop:singCos} to find implicit equations for ${\rm Sing}(\Xcos)$. Given a matrix $A = \left[ \begin{smallmatrix}
    1 & 0 & a \\
    0 & 1 & b
\end{smallmatrix} \right]$, for $a,b>0$, write $a = \frac{a_1}{a_2}$ and $b = \frac{b_1}{b_2}$ with $a_1, a_2 \in \mathbb{Z}$, resp.~$b_1, b_2 \in \mathbb{Z}$, coprime. Then, \eqref{eq:singCosPoints} parametrizes a subset of the vertices of $[-1,1]^3$ described by $\big( (-1)^{k_1}, (-1)^{k_2}, (-1)^{ak_1+bk_2} \big)$ for some $k_i \in \mathbb{Z}$. Equation \eqref{eq:singCosSelfInt} instead defines the curves
\begin{gather*}
    {\cal T}_{A,\cos}(H_{\{1\},\nu}) \, = \, \{x \in \mathbb{C}^3 \, : \, T_{a_1}(x_1) = (-1)^{\nu a_2}, \,  T_{b_1}(x_2) = T_{b_2}((-1)^{\nu a_2} x_3)\}, \\
    {\cal T}_{A,\cos}(H_{\{2\},\mu}) \, =  \,   \{x \in \mathbb{C}^3 \, : \, T_{b_1}(x_2) = (-1)^{\mu b_2}, \, T_{a_1}(x_1) = T_{a_2}((-1)^{\mu b_2} x_3)\}.
\end{gather*}
Here $\nu, \mu$ run over $\mathbb{Z}$, but only finitely many curves are distinct: they are obtained from $0\leq \nu \leq a_1$ and $0\leq \mu \leq b_1$.
Among these curves, the ones contained in the singular locus are $0 < \nu < a_1$ and $0 < \mu < b_1$. The extremal cases $\nu = 0,a_1$ and $\mu = 0, b_1$ do not produce singularities. In the notation of the proof of Proposition \ref{prop:singCos}, they violate the condition $u \neq \pm u' \mod 2\pi$. We illustrate the $3$-dimensional case with two examples.

\begin{example}
    Consider the matrix $A$ from Example \ref{ex:elliptope}. Its normal form is $\left[ \begin{smallmatrix}
    1 & 0 & 1 \\
    0 & 1 & 1
\end{smallmatrix} \right]$. The first type of singularities of $\Xcos$ are the points $(1, 1, 1) , (1, -1, -1), (-1, 1, -1), (-1, -1, 1)$. These the are the images of $L_{J,\kappa}$. The second type of singularities does not exist, since there are no integers in $(0,1)$. The images of $H_{I, \nu}$ for $I = \{1\}$ or $I = \{2\}$ and $\nu \in \mathbb{Z}$ are the smooth lines $\{x_1 \pm 1 = 0, x_2 \pm x_3 = 0\}$ and $\{x_2 \pm 1 = 0, x_1 \pm x_3 = 0\}$ contained in the elliptope.
\end{example}

\begin{example}\label{ex:trigSing}
    \begin{figure}[ht]
        \centering
        \includegraphics[height=5cm]{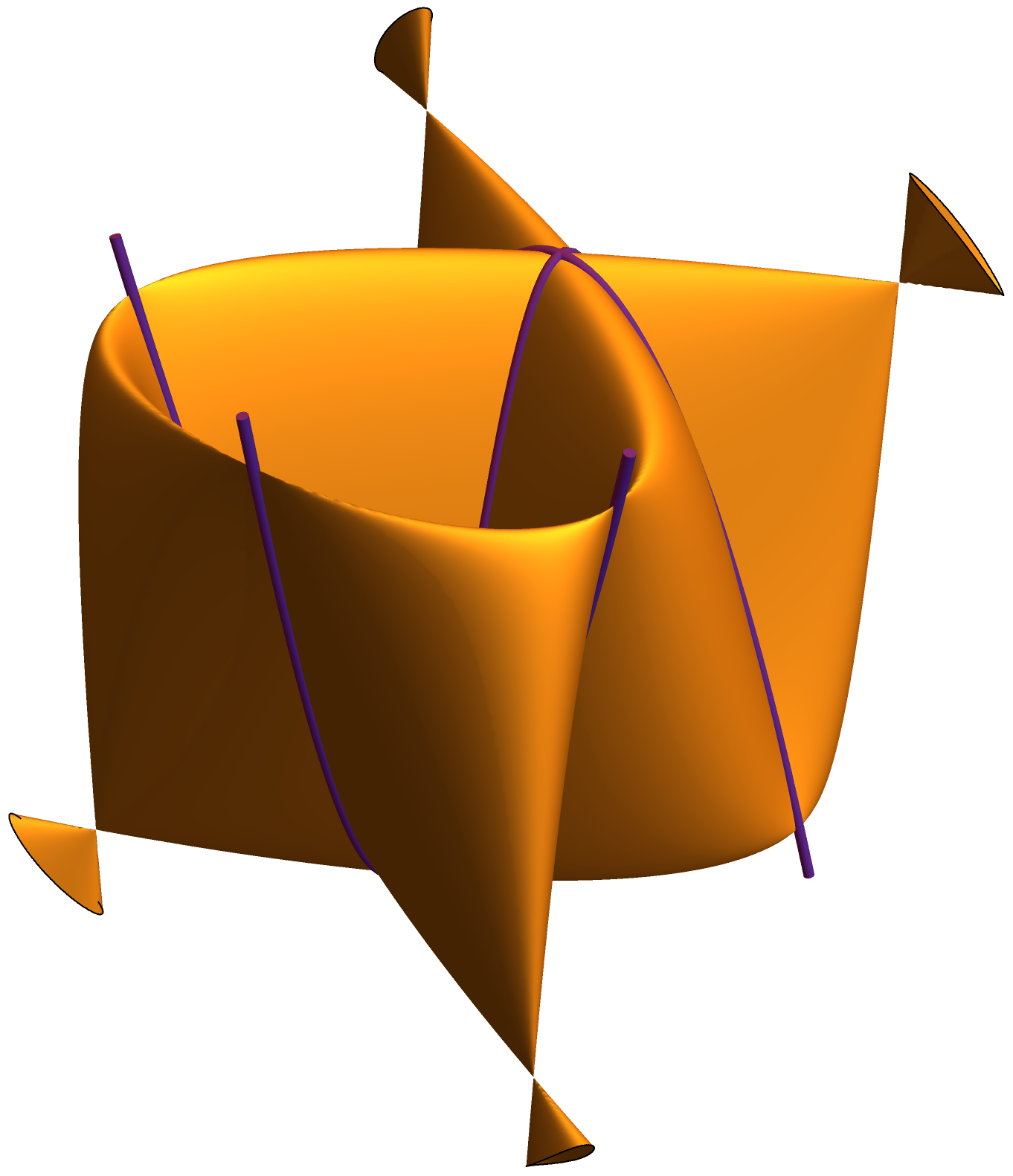}
        \quad \quad
        \includegraphics[height=5cm]{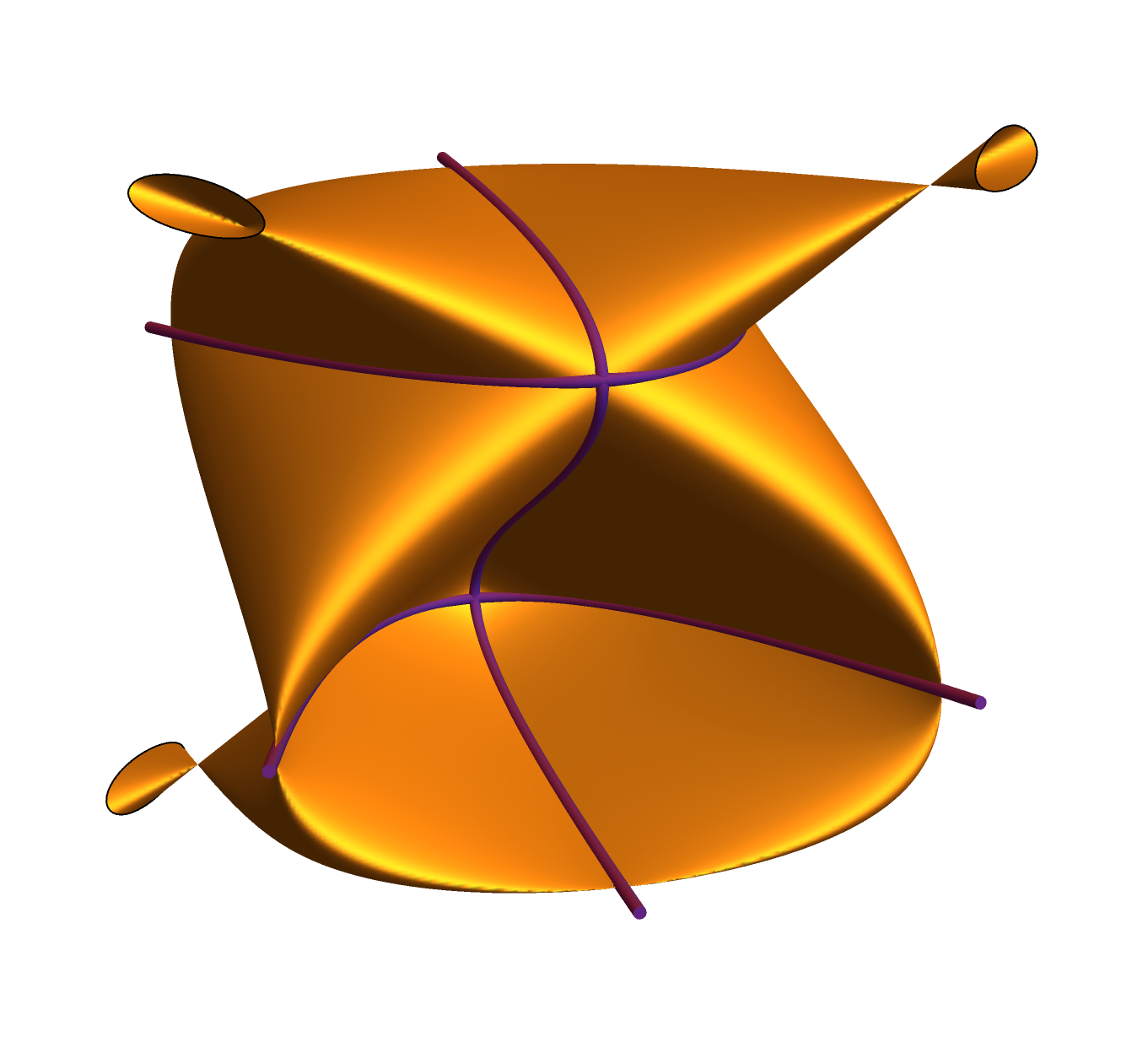}
        \caption{Two views on the Chebyshev surface from Example \ref{ex:trigSing} and its singular~curves.
        }
        \label{fig:trigSing}
    \end{figure}
    Consider the matrix $A = \left[ \begin{smallmatrix}
        1 & 0 & 2 \\
        0 & 1 & 3
    \end{smallmatrix} \right]$.
    The associated cosine Chebyshev variety $\Xcos$ is the sextic surface in Figure \ref{fig:trigSing}, defined by 
    \[
    4 x_1^4-16 x_1^2 x_2^3 x_3+12 x_1^2 x_2 x_3-4 x_1^2+16 x_2^6-24 x_2^4+8 x_2^3 x_3+9 x_2^2-6 x_2 x_3+x_3^2 = 0.
    \]
    Its singular locus consists of four points and three curves (purple):
    \begin{gather*}
        (1,1,1), \quad (-1,-1,-1), \quad (1,-1,1), \quad (-1,1,-1), \\
        \{x_1 = 0,\,  4 x_2^3 - 3 x_2  = -x_3 \}, \; \{ 2 x_2 = -1,\, 2 x_1^2 -1   = x_3 \}, \; \{ 2 x_2 = 1,\, 2 x_1^2 -1  = -x_3\}. \qedhere
    \end{gather*}
\end{example}

\section{Applications and future directions} \label{sec:6}

This section illustrates the role of Chebyshev varieties in some applications, and identifies some directions for future research. We start with the distribution of real roots of univariate polynomials which are sparse in the Chebyshev basis. Then, we illustrate an eigenvalue algorithm for solving Chebyshev equations in the case where $A$ satisfies the condition \eqref{eq:suffdense}. We solve linear equations on ${\cal X}_{A,\cos}$ using homotopy continuation. Finally, we discuss Chebyshev varieties parametrized by different multivariate Chebyshev polynomials from \cite{hubert2022sparse,ryland2010multivariate}. 

\subsection{Real roots}
    Theorems \ref{thm:hyperchev}, \ref{thm:hyperT} count the number of real roots of polynomials $\alpha T_a(t) - \beta T_b (t) = 0$. A natural question is what happens when we pass from Chebyshev plane curves to space curves. Consider $A = [2~3~7]$ and the Chebyshev $T$-curve ${\cal X}_{A,T}$ of degree $7$, displayed in Figure \ref{fig:space_curve}, left. 
    Given a plane $v^\perp\subset\mathbb{C}^3$, we are interested in the number of real points in ${\cal X}_{A,T}\cap v^\perp$. If the intersection is transversal, then locally the number is constant. It can change only when $v^\perp$ contains the tangent to ${\cal X}_{A,T}$ at a point in ${\cal X}_{A,T} \cap v^\perp$. This happens when $v$ satisfies 
    \small
    \begin{gather*}
        v_3 \cdot (2048 v_1^4 v_2^5+2304 v_1^2 v_2^7+27648 v_2^9+25000 v_1^8 v_3+28125 v_1^6 v_2^2 v_3+378460 v_1^4 v_2^4 v_3-26112 v_1^2 v_2^6 v_3\\
        -27648 v_2^8 v_3-481250 v_1^6 v_2 v_3^2-5797820 v_1^4 v_2^3 v_3^2+3930304 v_1^2 v_2^5 v_3^2+119808 v_2^7 v_3^2+153125 v_1^6 v_3^3\\
        +23852220 v_1^4 v_2^2 v_3^3-19302080 v_1^2 v_2^4 v_3^3+1480192 v_2^6 v_3^3-29985060 v_1^4 v_2 v_3^4+34354880 v_1^2 v_2^3 v_3^4\\
        -1229312 v_2^5 v_3^4+12850152 v_1^4 v_3^5-77561904 v_1^2 v_2^2 v_3^5+1229312 v_2^4 v_3^5+114556512 v_1^2 v_2 v_3^6\\
        +22588608 v_2^3 v_3^6-58353904 v_1^2 v_3^7-22588608 v_2^2 v_3^7-52706752 v_2 v_3^8+52706752 v_3^9)  \, = \, 0.
    \end{gather*}
    \normalsize
    This discriminant divides $\mathbb{R}^3$ up into chambers. For $v$ in the interior of such a chamber, there is a fixed number of real intersections.
    The chambers are invariant under scaling $v$. Hence, we may restrict to the upper hemisphere of $S^2$ without loss of generality. There are $5$ open chambers, see Figure \ref{fig:super}.
    \begin{figure}[ht]
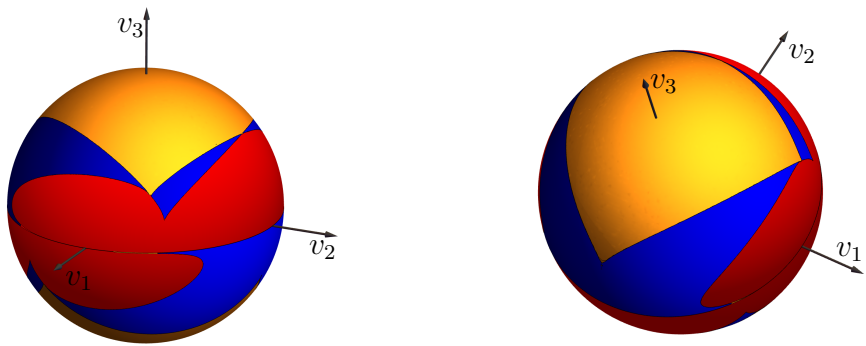

        \centering
        \input{Figures/super1}
        \qquad \qquad
        \input{Figures/super2}
        \caption{A discriminantal subdivision of the sphere: front and top view. }
        \label{fig:super}
    \end{figure}
    The yellow regions contain vectors $v$ such that ${\cal X}_{A,T} \cap v^\perp$ contains $7$ real points; for $v$ in the blue regions there are $5$ real and $2$ complex points; for $v$ in the red region there are $3$ real and $4$ complex points. On the coordinate plane $v_3=0$, there are $3$ real intersection points. At $v = (\pm 1,0,0)$, two out of $3$ real roots collide. There are always at least $2$ real solutions, and there are many planes that achieve the maximum number of $7$ real solutions.
This example and similar computations support the following conjecture.
\begin{conjecture}
    Let $A = [a_1~\cdots~a_n] \in \mathbb{N}^{1\times n}$ be such that 3 of its entries are pairwise coprime. Any hyperplane through the origin intersects ${\cal X}_{A,T}$ in at least $\min_{j \in [n]} a_j$ real points.
\end{conjecture}

\subsection{Solving tensor product Chebyshev equations} \label{sec:6-2}
Section \ref{sec:4} investigates Chebyshev varieties ${\cal X}_{A,\otimes}$ parametrized by tensor product Chebyshev polynomials ${\cal T}_{a_j,\otimes}(t_1, \ldots, t_m)$. Polynomial systems of equations of the form 
\begin{equation} \label{eq:tensoreqs}
f_i (t)  \, = \, c_{i,0} +  \sum_{j=1}^n c_{i,j} \, {\cal T}_{a_j,\otimes}(t)  \, = \, 0, \quad i = 1, \ldots, m
\end{equation}
are interpreted as linear equations on ${\cal X}_{A,\otimes}$. This suggests a strategy for solving \eqref{eq:tensoreqs}: (i) solve the linear system $c_0+C \cdot x  = 0$, and (ii) intersect the solution space with ${\cal X}_{A,\otimes}$. Step (i) is essentially a nullspace computation, and can be performed using (numerical) linear algebra. Step (ii) is less straightforward, but can be turned into linear algebra too. 

Here we describe a linear algebra algorithm based on this idea. It adapts well known strategies for the monomial basis (see e.g.~\cite{dreesen2012back,telen2020solving}) to that of tensor product Chebyshev polynomials. Importantly, throughout the outlined algorithm, we never expand into monomials.

For simplicity, we assume that the matrix $A$ has rank $m$ and that it is sufficiently dense: it satisfies \eqref{eq:suffdense}. Similar to \eqref{eq:tensormap}, for any set of multi-indices $A' = [a_1' ~ \cdots ~ a_{n'}'] \in \mathbb{N}^{m \times n'}$ we write ${\cal T}_{A', \otimes}(t)$ for the column vector $( {\cal T}_{a_1',\otimes}(t), \ldots, {\cal T}_{a_{n'}', \otimes}(t))^\top$.
Let $P_A = {\rm conv}(0 \cup A) \subset \mathbb{R}^m$ and write $\Delta_m = {\rm conv}(0, e_1, \ldots, e_m)$ for the standard simplex in $\mathbb{R}^m$. For $k \in \mathbb{N}$, the \emph{$k$-dilation} of $P_A$ is $k \cdot P_A = \{ k \cdot p  \, : \, p \in P_A \}$ and the \emph{Minkowski sum} of two polytopes $P, Q \subset \mathbb{R}^m$ is $P+Q = \{ p + q \, : \, p \in P, q \in Q\}$. We define the polytope 
\[ P_{A_{\rm ext}} \, = \, m \cdot P_A + \Delta_m \, = \, \{ m \cdot p + q \in \mathbb{R}^m \, : \, p \in P_A, \, q \in \Delta_m \}. \]
The lattice points $P_{A_{\rm ext}} \cap \mathbb{N}^m$ contained in this polytope are the columns of a matrix $A_{\rm ext}$. Similarly, the columns of a new matrix $B$ are the lattice points of $(m-1) \cdot P_A + \Delta_m$. With these definitions, one easily verifies that there is a unique matrix $M$ satisfying 
\[ M \cdot {\cal T}_{A_{\rm ext},\otimes}(t) \, = \, \begin{bmatrix}
    {\cal T}_{B,\otimes}(t) \cdot f_1(t) \\ 
    {\cal T}_{B,\otimes}(t) \cdot f_2(t) \\
    \vdots \\
    {\cal T}_{B,\otimes}(t) \cdot f_m(t)
\end{bmatrix}.\]
That matrix has size $(m \cdot |B|) \times |A_{\rm ext}|$, where $|\cdot|$ denotes the number of columns. Its rows are obtained by expanding ${\cal T}_{b, \otimes} (t)\cdot f_i(t)$ in the entries of ${\cal T}_{A_{\rm ext},\otimes}(t)$, for all $b \in B$ and $i = 1, \ldots, m$. We write $V_{\mathbb{C}^m}(f_1, \ldots, f_m) = \{ t^{(1)}, \ldots, t^{(\delta)}\}$ for the set of complex solutions to our equations.

\begin{theorem} \label{thm:nullspacebasis}
    Let $A \in \mathbb{N}^{m \times n}$ be such that ${\rm rank}(A) = m<n$ and it satisfies \eqref{eq:suffdense}. For generic $C, c_0$, the vectors ${\cal T}_{A_{\rm ext},\otimes}(t^{(i)})$, $i = 1, \ldots, \delta$ form a basis for the kernel of the matrix $M$ defined above. In particular, $M$ has rank $|A_{\rm ext}| - \delta = |A_{\rm ext}| - m! \cdot {\rm vol}(P_A)$.
\end{theorem}
\begin{proof}
    It is clear that $M \cdot {\cal T}_{A_{\rm ext},\otimes}(t^{(i)}) = 0$ for $i = 1, \ldots, \delta$. The fact that these vectors form a basis for $\ker M$ is a consequence of \cite[Theorem 4.4]{bender2022toric}. The statement about the rank follows from this fact and Theorem \ref{thm:suffdense}: by genericity of $C, c_0$, we have $\delta = \deg {\cal X}_{A,\otimes}$.
\end{proof}
Here is a reformulation of Theorem \ref{thm:nullspacebasis} in terms of Chebyshev varieties.
\begin{corollary}
   In the situation of Theorem \ref{thm:nullspacebasis}, the linear space $\ker M \subset \mathbb{C}^{|A_{\rm ext}|}$ intersects the Chebyshev variety ${\cal X}_{A_{\rm ext}, \otimes}$ in $\delta$ points, namely ${\cal T}_{A_{\rm ext},\otimes}(t^{(i)}) = 0$ for $i = 1, \ldots, \delta$. These points are independent vectors in $\mathbb{C}^{|A_{\rm ext}|}$.
\end{corollary}

Let $N = [{\cal T}_{A_{\rm ext}, \otimes}(t^{(1)}) ~ \cdots ~ {\cal T}_{A_{\rm ext}, \otimes}(t^{(\delta)})] \in \mathbb{C}^{|A_{\rm ext}|\times \delta}$ be the nullspace matrix of $M$ corresponding to the basis from Theorem \ref{thm:nullspacebasis}. Its rows are indexed by the columns of $A_{\rm ext}$. We consider the submatrix $N_{A_0}$ consisting of the rows indexed by $A_0 = m \cdot P_A \cap \mathbb{N}^m \subset A_{\rm ext}$.

For $i = 1, \ldots, m$, there is a unique matrix $C_i \in \mathbb{R}^{|A_0| \times |A_{\rm ext}|}$ satisfying the relation 
\[ t_i \cdot {\cal T}_{A_0, \otimes}(t) \, = \, C_i \cdot {\cal T}_{A_{\rm ext}, \otimes}(t). \]
To construct $C_i$ we use the recurrence \eqref{eq:chebpol1}. This translates into the following matrix relation: 
\begin{equation} \label{eq:EVrel0}
 N_{A_0} \cdot {\rm diag}(t_i^{(1)}, t_i^{(2)}, \ldots, t_i^{(\delta)}) \, = \, C_i \cdot N, \quad i = 1, \ldots, m.\end{equation}
 Another consequence of \cite[Theorem 4.4]{bender2022toric} is that, in the situation of Theorem \ref{thm:nullspacebasis}, we have ${\rm rank}(N_{A_0}) = \delta = m! \cdot {\rm vol}(P_A)$. Therefore, $N_{A_0}$ has a left pseudo-inverse $N_{A_0}^\dag$ such that $N_{A_0}^\dag N_{A_0}$ is the identity matrix of size $\delta$. Applying this from the left to \eqref{eq:EVrel0} we obtain 
 \[  {\rm diag}(t_i^{(1)}, t_i^{(2)}, \ldots, t_i^{(\delta)}) \, = \, N_{A_0}^\dag \cdot C_i \cdot N. \]
 An arbitrary kernel matrix for $M$ can be written as $\tilde{N} = N \cdot V$, where $V \in \mathbb{C}^{\delta \times \delta}$ is a square invertible matrix. In practice, we compute such a matrix $\tilde{N} = N \cdot V$ and its submatrix $\tilde{N}_{A_0} = N_{A_0} \cdot V$ for some $V$, using for instance the SVD. Equation \eqref{eq:EVrel0} can be written as
 \[ (N_{A_0} \cdot V) \cdot V^{-1} \cdot {\rm diag}(t_i^{(1)}, t_i^{(2)}, \ldots, t_i^{(\delta)}) \, = \, C_i \cdot (N \cdot V) \cdot V^{-1}  \]
 and this implies $V^{-1} \cdot {\rm diag}(t_i^{(1)}, t_i^{(2)}, \ldots, t_i^{(\delta)}) \cdot V \, = \, (N_{A_0} \cdot V)^\dag \cdot C_i \cdot (N \cdot V)$. We have now shown that the $i^{th}$ coordinates of the roots of $f_1 = \cdots = f_m = 0$ are the eigenvalues of $\tilde{N}_{A_0}^\dag \cdot C_i \cdot \tilde{N}$. Notice that the eigenvector of $t_i^{(j)}$ does not depend on the variable index $i$. This helps to group the eigenvalues of the $m$ matrices $\tilde{N}_{A_0}^\dag \cdot C_i \cdot \tilde{N}$ into coordinates of the roots $t^{(j)}$. For more information, see for instance \cite[p.~111]{telen2020solving}. An efficient implementation and analysis of this algorithm is left for future research. It would be especially interesting to understand its stability for root finding in the box $\square = [-1,1]^m$, under some assumptions on $f_i$. Here, we limit ourselves to presenting some results of a proof-of-concept implementation. 
\begin{example}\label{ex:tensor_curves_ev}
    Let $A$ be the matrix of all tuples $a_j$ of Euclidean degree $k = 30$ (see Remark \ref{rem:densechoices}). Using the Julia package \texttt{Oscar.jl} \cite{OSCAR}, we compute that $\delta = 2! \cdot {\rm vol}(P_A) = 1396$. We set up the system \eqref{eq:tensoreqs} with real coefficients $c_{i,j}, c_{i,0}$ drawn from a standard normal distribution, and use the outlined eigenvalue algorithm to solve it. The matrix $M$ has size $1560 \times 2953$. Its (numerical) rank is 1557. Among the 1396 complex (approximate) solutions, 382 are (approximately) real, and 338 are contained in the square $[-1,1]^2$, see Figure \ref{fig:eucldeg30}.
    \begin{figure}
        \centering
        \includegraphics[height = 5.7cm]{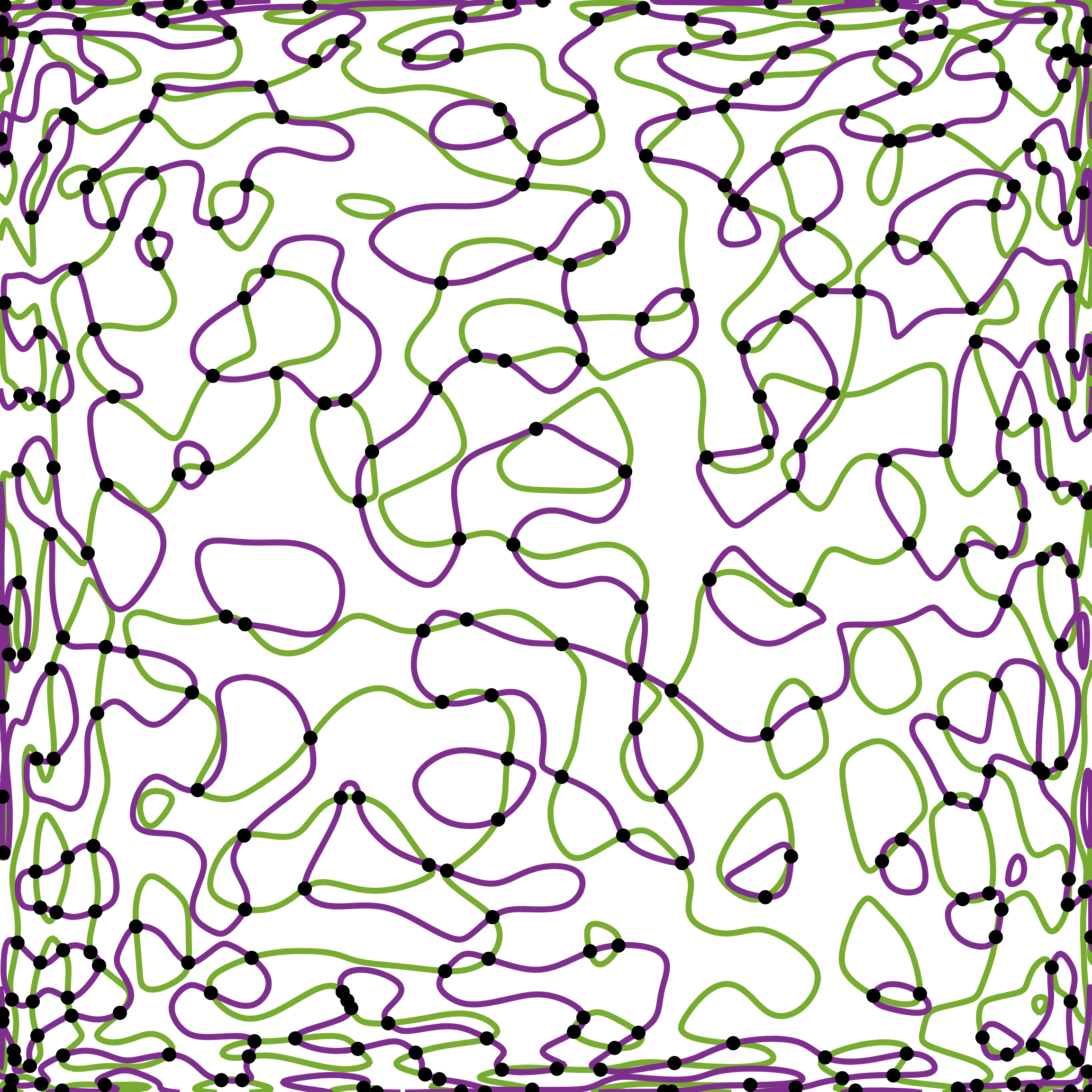}
        \caption{The curves from Example \ref{ex:tensor_curves_ev} in $[-1,1]^2$.}
        \label{fig:eucldeg30}
    \end{figure}
\end{example}

\subsection{Solving cosine equations} \label{sec:6-3}
We use the results from Section \ref{sec:5} to solve the system of equations
    \begin{equation}\label{eq:cos_system}
    c_{i,0} + \sum_{j=1}^n c_{i,j} \cos(a_j \cdot u)  \, = \, 0, \qquad i = 1, \ldots, m,
    \end{equation}
via homotopy continuation methods \cite{sommese2005numerical}. Changing coordinates $v = e^{\sqrt{-1}u}$, we obtain
    \begin{equation}\label{eq:cos_system_v}
    f_i(v; C, c_0) \, = \, c_{i,0} + \sum_{j=1}^n c_{i,j} \left ( \frac{v^{a_j} + v^{-a_j}}{2}\right ) \,= \,  0, \qquad i = 1, \ldots, m.
    \end{equation}
We interpret the coefficients $C, c_0$ as complex parameters, and apply the method of \emph{monodromy loops} to solve \eqref{eq:cos_system_v} for specific values $C_1 \in \mathbb{R}^{m \times n}, c_{01} \in \mathbb{R}^m$ of these parameters \cite{duff2018monodromy}. We explain how this works in a nutshell. Pick a random $v_0 \in (\mathbb{C} \setminus \{0\})^m$, and find parameters $C_0 \in \mathbb{C}^{m,n}, c_{00} \in \mathbb{C}^m$ for which $f_i(v_0; C_0, c_{00}) = 0, i = 1, \ldots, m$. This is done by solving a linear system of equations. Next, use monodromy loops in the space of parameters $C, c_0$ to compute all solutions of $f_i(v; C_0, c_{00}) = 0, i = 1, \ldots, m$. Since we are interested in the solutions for the parameter values $C_1, c_{01}$, we set up the homotopy
\begin{equation} \label{eq:homotopy} f_i \left ( \, v \, ; \, C_0\tau + C_1(1-\tau), \, c_{00}\tau  + c_{01}(1-\tau) \right ) = 0, \quad i = 1, \ldots, m, \quad  \tau \in [0,1]. \end{equation}
We track the solutions from $\tau = 0$ to $\tau = 1$. When implementing this algorithm, it is important to make use of the fact that solutions come in orbits of a group action. Namely, for each solution $v$, the coordinatewise inverse $v^{-1}$ is also a solution. Also, it is essential that we know when to stop the monodromy procedure: the number of solution orbits for generic parameters $C_0, c_{00}$ is predicted by Theorem \ref{thm:degCos}. 
We illustrate this with an example.

\begin{example}[$m = 2, n = 6$]\label{ex:cos_eqns}
    Let $A = \left[ \begin{smallmatrix}
     4 & 4 & 6 & 7 & 9 & 2 \\
     8 & 4 & 1 & 2 & 6 & 7
    \end{smallmatrix} \right]$, $C_1 = \left[ \begin{smallmatrix}
     1 & 2 & 3 & 5 & -1 & -7 \\
     -2 & -6 & 5 & -3 & 1 & 4
    \end{smallmatrix} \right]$ and $c_{01} = \left[ \begin{smallmatrix}
        4\\ -2
    \end{smallmatrix} \right ]$.
    From $A$, we construct the cosine Chebyshev surface $\Xcos\subset \mathbb{C}^6$. According to Theorem \ref{thm:degCos}, its degree is $129$. This is the area of $P_{A,\cos}$, which is the octagon shown in Figure \ref{fig:cos_system}, left. Hence, if $C_0 \in \mathbb{C}^{2\times 7}$ and $c_{00} \in \mathbb{C}^2$ are sufficiently generic, the number of solution pairs $(v,v^{-1})$ to the system \eqref{eq:cos_system} is $129$. We compute these with the following Julia code snippet:

    \footnotesize
\begin{lstlisting}[columns=fullflexible]
using HomotopyContinuation, Oscar   # load packages
A = [4 4 6 7 9 2; 8 4 1 2 6 7]; m,n = size(A)
C1 = [1 2 3 5 -1 -7; -2 -6 5 -3 1 4]; c01 = [4, -2];    # target parameters
P = convex_hull(transpose([A -A]))      # compute the polytope P_{A,cos}
deg = factorial(m)*volume(P)/2        # degree formula from Theorem 5.3
@var v[1:m] C[1:m,1:n] c0[1:m]  # declare variables and parameters
eqs = C*[1/2*(prod(v.^A[:,j])+prod(v.^(-A[:,j]))) for j = 1:n]+c0
F = System(eqs, variables = v, parameters = [C[:];c0])
act(p) = [pp^(-1) for pp in p] # coordinate-wise inverse acts on solutions
# The following command uses monodromy loops to solve for generic parameters C0, c00 
monres = monodromy_solve(F, group_action = act; target_solutions_count = Int(deg)) 
# The next line sets up the homotopy in Eq. (25) and tracks from tau = 0 to tau = 1.
R = HomotopyContinuation.solve(F,solutions(monres); 
     start_parameters = parameters(monres), target_parameters = [C1[:];c01])
u_solutions = [-im*log.(vv) for vv in solutions(R)] # from v to u coords
\end{lstlisting}
\normalsize
On a MacBook Pro machine with a Quad-Core Intel Core i7 processor running at 2.8 GHz, it takes about 0.2 seconds to run this code, after loading packages and after one compile run. We use Julia v1.9.1, \texttt{HomotopyContinuation.jl} v2.9.2 \cite{HomotopyContinuation.jl}, and \texttt{Oscar.jl} v0.13.0 \cite{OSCAR}.

Among the $129$ pairs of solutions $(v,v^{-1})$, $5$ are real. After changing coordinates, \texttt{u\_solutions} contains $64$ pairs $(u,-u)$ of real solutions (modulo $2\pi$) and $65$ pairs of complex solutions (modulo $2\pi$). The real solutions are the $128$ black dots displayed in Figure \ref{fig:cos_system}, right.
The green and purple curves in Figure \ref{fig:cos_system} are given by the equations in \eqref{eq:cos_system}.
\begin{figure}[ht]
    \centering
    \begin{tikzpicture}[scale=1.2]
\begin{axis}[%
width=1.9in,
height=1.9in,
scale only axis,
xmin=-10,
xmax=10,
ymin=-10,
ymax=10,
ticks = none, 
ticks = none,
axis background/.style={fill=white},
axis line style={draw=none} 
]

\draw[->,>=stealth] (axis cs:0,0) -- (axis cs:9.7,0);
\draw[->,>=stealth] (axis cs:0,0) -- (axis cs:0,8.7);

\addplot[only marks,mark=*,mark size=0.5pt,black
        ]  coordinates {
(-9,-8) (-9,-7) (-9,-6) (-9,-5) (-9,-4) (-9,-3) (-9,-2) (-9,-1) (-9,0)
(-9,8) (-9,7) (-9,6) (-9,5) (-9,4) (-9,3) (-9,2) (-9,1)
(-8,-8) (-8,-7) (-8,-6) (-8,-5) (-8,-4) (-8,-3) (-8,-2) (-8,-1) (-8,0) (-8,8) (-8,7) (-8,6) (-8,5) (-8,4) (-8,3) (-8,2) (-8,1)
(-7,-8) (-7,-7) (-7,-6) (-7,-5) (-7,-4) (-7,-3) (-7,-2) (-7,-1) (-7,0) (-7,8) (-7,7) (-7,6) (-7,5) (-7,4) (-7,3) (-7,2) (-7,1)
(-6,-8) (-6,-7) (-6,-6) (-6,-5) (-6,-4) (-6,-3) (-6,-2) (-6,-1) (-6,0) (-6,8) (-6,7) (-6,6) (-6,5) (-6,4) (-6,3) (-6,2) (-6,1)
(-5,-8) (-5,-7) (-5,-6) (-5,-5) (-5,-4) (-5,-3) (-5,-2) (-5,-1) (-5,0) (-5,8) (-5,7) (-5,6) (-5,5) (-5,4) (-5,3) (-5,2) (-5,1)
(-4,-8) (-4,-7) (-4,-6) (-4,-5) (-4,-4) (-4,-3) (-4,-2) (-4,-1) (-4,0) (-4,8) (-4,7) (-4,6) (-4,5) (-4,4) (-4,3) (-4,2) (-4,1)
(-3,-8) (-3,-7) (-3,-6) (-3,-5) (-3,-4) (-3,-3) (-3,-2) (-3,-1) (-3,0) (-3,8) (-3,7) (-3,6) (-3,5) (-3,4) (-3,3) (-3,2) (-3,1)
(-2,-8) (-2,-7) (-2,-6) (-2,-5) (-2,-4) (-2,-3) (-2,-2) (-2,-1) (-2,0) (-2,8) (-2,7) (-2,6) (-2,5) (-2,4) (-2,3) (-2,2) (-2,1)
(-1,-8) (-1,-7) (-1,-6) (-1,-5) (-1,-4) (-1,-3) (-1,-2) (-1,-1) (-1,0) (-1,8) (-1,7) (-1,6) (-1,5) (-1,4) (-1,3) (-1,2) (-1,1)
(0,-8) (0,-7) (0,-6) (0,-5) (0,-4) (0,-3) (0,-2) (0,-1) (0,0) (0,8) (0,7) (0,6) (0,5) (0,4) (0,3) (0,2) (0,1)
(1,-8) (1,-7) (1,-6) (1,-5) (1,-4) (1,-3) (1,-2) (1,-1) (1,0) (1,8) (1,7) (1,6) (1,5) (1,4) (1,3) (1,2) (1,1)
(2,-8) (2,-7) (2,-6) (2,-5) (2,-4) (2,-3) (2,-2) (2,-1) (2,0) (2,8) (2,7) (2,6) (2,5) (2,4) (2,3) (2,2) (2,1)
(3,-8) (3,-7) (3,-6) (3,-5) (3,-4) (3,-3) (3,-2) (3,-1) (3,0) (3,8) (3,7) (3,6) (3,5) (3,4) (3,3) (3,2) (3,1)
(4,-8) (4,-7) (4,-6) (4,-5) (4,-4) (4,-3) (4,-2) (4,-1) (4,0) (4,8) (4,7) (4,6) (4,5) (4,4) (4,3) (4,2) (4,1)
(5,-8) (5,-7) (5,-6) (5,-5) (5,-4) (5,-3) (5,-2) (5,-1) (5,0) (5,8) (5,7) (5,6) (5,5) (5,4) (5,3) (5,2) (5,1)
(6,-8) (6,-7) (6,-6) (6,-5) (6,-4) (6,-3) (6,-2) (6,-1) (6,0) (6,8) (6,7) (6,6) (6,5) (6,4) (6,3) (6,2) (6,1)
(7,-8) (7,-7) (7,-6) (7,-5) (7,-4) (7,-3) (7,-2) (7,-1) (7,0) (7,8) (7,7) (7,6) (7,5) (7,4) (7,3) (7,2) (7,1)
(8,-8) (8,-7) (8,-6) (8,-5) (8,-4) (8,-3) (8,-2) (8,-1) (8,0) (8,8) (8,7) (8,6) (8,5) (8,4) (8,3) (8,2) (8,1)
(9,-8) (9,-7) (9,-6) (9,-5) (9,-4) (9,-3) (9,-2) (9,-1) (9,0) (9,8) (9,7) (9,6) (9,5) (9,4) (9,3) (9,2) (9,1)
};

\addplot [color=myblue,solid,fill opacity=0.2,fill = myblue,forget plot]
  table[row sep=crcr]{%
7 2\\
9 6\\
4 8\\
2 7\\
-7 -2\\
-9 -6\\
-4 -8\\
-2 -7\\
7 2\\
};

\addplot[only marks,mark=*,mark size=2.1pt,myblue
        ]  coordinates {
  (7, 2) 
(9, 6)
(4, 8)
(2, 7)
(-7, -2)
(-9, -6)
(-4, -8)
(-2, -7)
};

\end{axis}
\end{tikzpicture} 
    \quad \quad \quad 
    \includegraphics[height = 5.7cm]{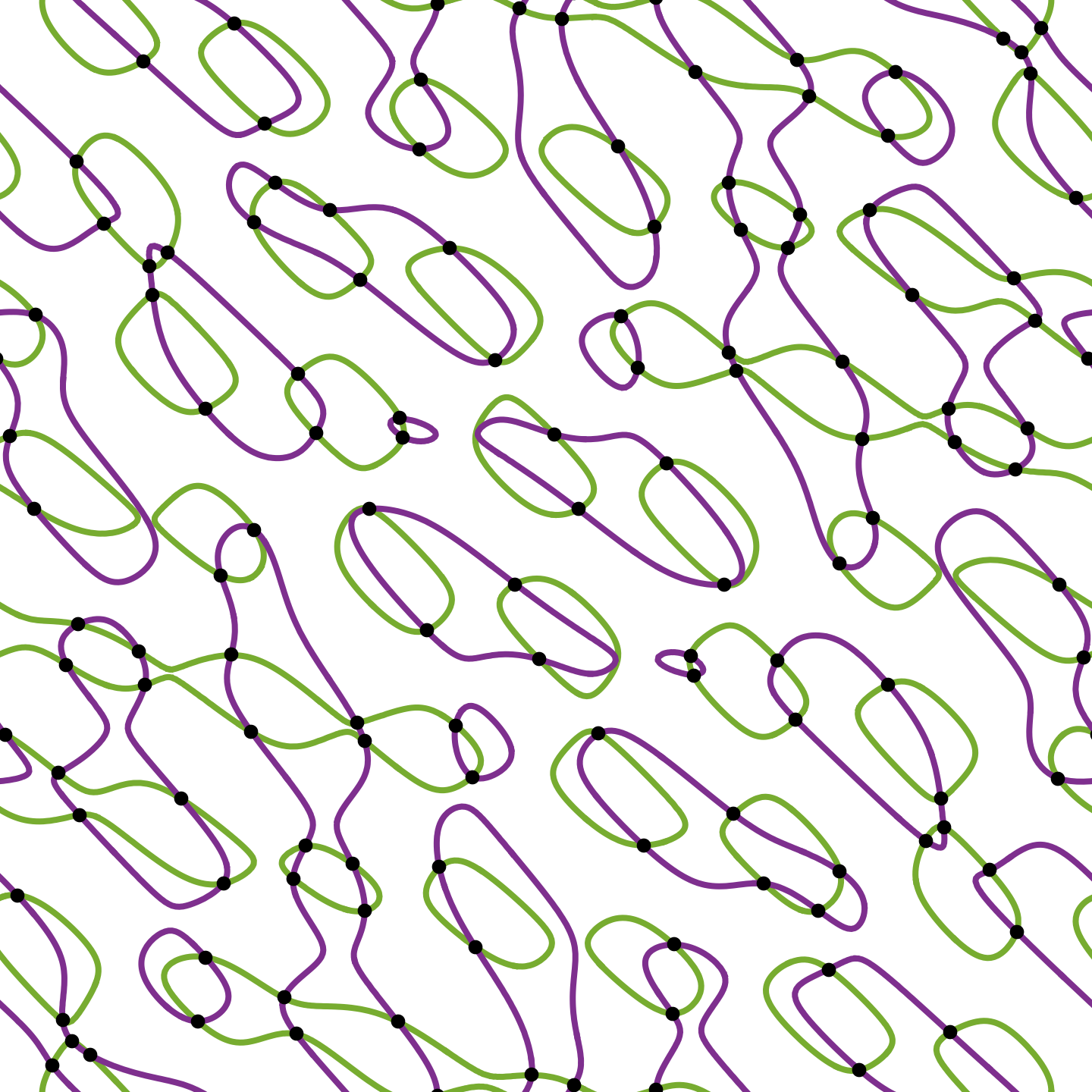}
    \caption{Left: the octagon $P_{A,\cos}$. Right: the curves from Example \ref{ex:cos_eqns} in $[0,2\pi]^2$.}
    \label{fig:cos_system}
\end{figure}
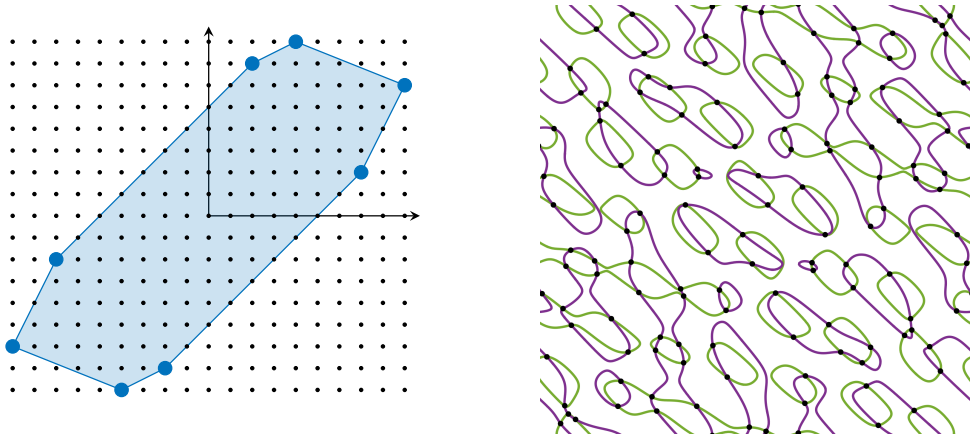
\end{example}

\subsection{Root systems} \label{sec:6-4}

Finally, we consider yet another notion of multivariate Chebyshev functions. We briefly present the construction, which uses root systems and Weyl groups. The details can be found in \cite{hubert2022sparse,ryland2010multivariate}. We have seen that univariate Chebyshev polynomials can be characterized~by 
\[ 
T_{k} \left (\frac{x + x^{-1}}{2} \right ) \,  = \,  \frac{x^k + x^{-k}}{2} \quad \text{for} \quad x \neq 0. 
\]
Consider the group of symmetries $\mathcal{W} = \{[-1], [1]\}$ on $\mathbb{R}$. We define an action on Laurent monomials by $[-1] \cdot x^k = x^{-k}$ and $[1] \cdot x^k = x^k$. By linearity, this defines an action on Laurent polynomials $\mathbb{Z}[x^{\pm 1}]$. 
One can prove that the ring $\mathbb{Z}[x^{\pm 1}]^{ \mathcal{W}}$ of $\mathcal{W}$-invariants is generated by the polynomial $\Theta_\omega = x + x^{-1}$, i.e., $\mathbb{Z}[x^{\pm 1}]^{ \mathcal{W}} = \mathbb{Z}[x + x^{-1}]$. The \textit{orbit polynomial} $\Theta_k(x) = 
x^k + x^{-k}$ is invariant under the action of $\mathcal{W}$ and can thus be written as $\Tilde{T}_k(x+x^{-1})$, where $\Tilde{T}_k \in \mathbb{Z}[t]$. 
The Chebyshev polynomial $T_k$ can be recovered as $T_k(t) = \frac{1}{2}  \Tilde{T}_k(2t)$. 

We now generalize this definition to multivariate polynomials following \cite{hubert2022sparse,ryland2010multivariate}. 
A root system $\Phi \subset \mathbb{R}^m$ defines a \emph{weight lattice} $\Gamma = \mathbb{Z}\omega_1 \oplus \ldots \oplus \mathbb{Z} \omega_m $, invariant under the action of its \emph{Weyl group} $\mathcal{W}$. We identify $\Gamma$ with $\mathbb{Z}^m$, and ${\cal W}$ can be seen as a subgroup of $GL_m(\mathbb{Z})$. The orbit polynomials $\Theta_{\omega_1}, \ldots, \Theta_{\omega_m}$, where $\Theta_{\omega_i} = \sum_{B \in \mathcal{W}} x^{B \omega_i}$, are algebraically independent and generate the ring of $\mathcal{W}$-invariant Laurent polynomials, i.e., $\mathbb{Z}[x^{\pm 1}]^{\mathcal{W}}=\mathbb{Z}[\Theta_{\omega_1},\ldots,\Theta_{\omega_m}]
$, see \cite[Chapitre VI, §3.3 Th\'eor\`eme 1]{bourbakiroots}. For $\alpha \in \mathbb{N}^m$, the orbit polynomial
\[ 
\Theta_{\alpha} \, = \, \sum_{B \in \mathcal{W}} x^{B \alpha}
\]
is $\mathcal{W}$-invariant. Therefore, there exists a unique polynomial ${\cal T}_{\alpha}~\in~\mathbb{Z}[t_1,\ldots,t_m]$ such that $\Theta_\alpha = {\cal T}_{\alpha}(\Theta_{
\omega_1}, \ldots, \Theta_{\omega_m})$, and it is defined to be the $\alpha^{th}$ generalized Chebyshev polynomial. Note, however, that the Weyl group $\mathcal{W} = \{[-1], [1]\}$ only leads to the classical Chebyshev polynomial up to normalization. 
To recover the standard definition in the univariate case, one can use normalized orbit polynomials instead. These are given by $\frac{1}{|{\cal W}|}\sum_{B \in {\cal W}} x^{B\alpha}$.
We choose not to include the normalizing factor in our definition to be coherent with \cite{hubert2022sparse}.

\begin{example}[$m = 2$]\label{ex:root_systems}
    Consider the root system ${\cal A}_2 = \{\pm [2, -1]^\top, \pm [-1, 2]^\top, \pm [1, 1]^\top  \}$. The basis of $\Gamma$ can be chosen to be $\omega_1 = [1, 0]^\top$ and $\omega_2 = [0, 1]^\top$. The Weyl group is a subgroup of $GL_2(\mathbb{Z})$ of order $6$. Its elements are displayed in \cite[Example 2.9]{hubert2022sparse}. The generalized Chebyshev polynomials satisfy the recurrence relations
    \begin{gather*}
        {\cal T}_{0,0} = 6, \quad {\cal T}_{1,0} = x, \quad {\cal T}_{0,1} = y, \quad {\cal T}_{1,1} = \frac{1}{4} xy - 3, \\
        {\cal T}_{a+2,0} = \frac{1}{2} \,x\, {\cal T}_{a+1,0} - 2 {\cal T}_{a,1}, \quad {\cal T}_{0,b+2} = \frac{1}{2} \,y\, {\cal T}_{0,b+1} - 2 {\cal T}_{1,b}, \\
        {\cal T}_{a+1,b} = \frac{1}{2} \,x\, {\cal T}_{a,b} - {\cal T}_{a-1,b+1} - {\cal T}_{a,b-1}, \quad {\cal T}_{a,b+1} = \frac{1}{2} \,y\, {\cal T}_{a,b} - {\cal T}_{a+1,b-1} - {\cal T}_{a-1,b}.
    \end{gather*}
    We compute the surface defined by the matrix $A = \left[ \begin{smallmatrix}
        1 & 1 & 2 \\
        2 & 1 & 3
    \end{smallmatrix}\right]$ from our running example with this choice of multivariate Chebyshev polynomials. The parametrization sends $(t_1, t_2)$ to 
    \[
    \left( \frac{1}{8}(  t_2^2 t_1 - 4 t_1^2-4 t_2), \frac{1}{4} (t_1 t_2 - 12), \frac{1}{128}(-9 t_1^3 t_2+4 t_1^2 t_2^3+96 t_1^2+36 t_1 t_2^2-t_2^4-268 t_2) \right).
    \]
    The Zariski closure of the image
    is an irreducible surface of degree $11$, defined by
    \begin{gather*}
    686 y^{11}+1296 x^3 y^7
     + \,  \ldots \, -\,65536 z^3+4622782052 y^2+2511433848 y+600952464 \, = \,  0.\qedhere
    \end{gather*}
\end{example}

\section*{Acknowledgements}
We would like to thank Evelyne Hubert, Yuji Nakatsukasa, Vanni Noferini, Frank Sottile, Alex Townsend and Nick Trefethen for useful discussions. 

\small
\bibliographystyle{abbrv}
\bibliography{references.bib}

\noindent{\bf Authors' addresses:}
\medskip

\noindent Za\"ineb Bel-Afia, Oxford University \hfill{\tt belafia@maths.ox.ac.uk}

\noindent Chiara Meroni, Harvard University \hfill{\tt cmeroni@seas.harvard.edu}

\noindent Simon Telen, MPI-MiS Leipzig
\hfill {\tt simon.telen@mis.mpg.de}
\end{document}